\documentclass[11pt,reqno]{amsart}
\usepackage[T1]{fontenc}
\usepackage{amsmath}
\usepackage{amsthm}
\usepackage{amsfonts}
\usepackage{amssymb}
\usepackage{graphicx}
\usepackage{amsbsy}
\usepackage{mathrsfs}
\usepackage{bbm}
\usepackage{enumerate}
\newtheorem{theorem}{Theorem}[section]

\newtheorem{lemma}[theorem]{Lemma}

\newtheorem{proposition}[theorem]{Proposition}
\newtheorem{claim}{Claim}

\theoremstyle{remark}
\newtheorem{remark}[theorem]{\it \bf{Remark}\/}

\numberwithin{equation}{section}
\catcode`@=11
\def\section{\@startsection{section}{1}%
  \z@{1.5\linespacing\@plus\linespacing}{.5\linespacing}%
  {\normalfont\bfseries\large\centering}}
\catcode`@=12
\newcommand{\be}{\begin{equation}}
\newcommand{\ee}{\end{equation}}
\newcommand{\bea}{\begin{eqnarray}}
\newcommand{\eea}{\end{eqnarray}}
\newcommand{\bee}{\begin{eqnarray*}}
\newcommand{\eee}{\end{eqnarray*}}

\def\RR{\mathbb{R}}

\catcode`@=11
\def\supess{\mathop{\operator@font Sup\,ess}}
\catcode`@=12

\def\RR{\mathbb{R}}

\def\bar#1{{\overline #1}}

\def\R2+{\RR ^2_+}

\def\lim{\mathop{\rm lim}}

\def\sup{\mathop{\rm sup}}
\def\exp{{\rm exp}}

\title[Multi-solitons for quartic gKdV]{On the nonexistence of pure multi-solitons for the quartic gKdV equation}
\author[Y. Martel]{Yvan Martel}
\address{Ecole Polytechnique, CMLS  CNRS UMR7640}
\email{yvan.martel@polytechnique.edu}
\author[F. Merle]{Frank Merle}
\address{Universit\'e de Cergy Pontoise and Institut des Hautes \'Etudes Scientifiques, AGM CNRS UMR8088}
\email{merle@math.u-cergy.fr}
 
\begin{document}

\begin{abstract}
We consider the quartic (nonintegrable) (gKdV) equation
$$
 \partial_t u + \partial_x (\partial_x^2 u   + u^4) =0,\quad t,x\in \RR,
$$
and $u(t)$ an \emph{outgoing} $2$-soliton of the equation,
i.e. a solution satisfying
$$	\lim_{t\to +\infty} \big\|u(t)- Q_{c_1} (. - c_1 t)-Q_{c_2}(.-c_2t)\big\|_{H^1} =0,
$$
where $0<c_2<c_1$ and where $Q_{c_j}(x-c_jt)$ are explicit solitons of the equation.

In \cite{MMalmost}, in the case $ 0< 1- \frac {c_2}{c_1}= \epsilon\leq \epsilon_0$, where $\epsilon_0$ is a small enough, not explicit constant,  the solution $u(t,x)$ is computed up to some order of $\epsilon$, for all  $t$ and  $x$. In particular, it is deduced that $u(t)$ is \emph{not} a multi-soliton as $t\to -\infty$,
proving the nonexistence of pure multi-soliton in this context.

In the present paper, we prove the same result for an explicit  range of speeds: $\frac 34 c_1< {c_2}< c_1$, by a different approach, which does not longer require a precise description of the solution. In fact, the nonexistence result holds for outgoing $N$-solitons, for any $N\geq 2$, under the assumption:
$\sum_{j=2}^N \left( 1- {c_j}/{c_1}\right)^2 \leq \frac 1{16}$ which is a natural generalization of  the condition for $N=2$.
\end{abstract}

\maketitle
\section{Introduction}
\subsection{Setting of the problem}
 In this paper, we focus on the quartic generalized Korteweg-de Vries (gKdV) equation
\begin{equation}\label{kdv}
 \partial_t u + \partial_x (\partial_x^2 u   + u^4) =0,\quad t,x\in \RR. 
\end{equation}
Recall  that  the Cauchy problem for \eqref{kdv} is globally well-posed in $H^1$ (see Kenig, Ponce and Vega \cite{KPV} for a precise existence and uniqueness statement), and that any $H^1$ solution $u(t,x)$ of \eqref{kdv} satisfies for all $t\in \RR$,
\begin{align}
	&	\int u^2(t) = M(u(t)) = M(u(0)) \qquad \text{(mass)}	\label{mass}\\
	&	\int (\partial_x u)^2 (t)  -\frac 25 u^5(t)  =  {E}(u(t)) =  {E}(u(0))\qquad \text{(energy)}	\label{energy}
\end{align}
Recall also that the integral of $u(t)$ is   preserved provided it is well-defined:
\be\label{int}
\int u(t)=\int u(0).
\ee
We call \emph{soliton} a solution of \eqref{kdv} of the form
$$R_{c,y_0}(t,x)= Q_c(x-ct-y_0),\quad \hbox{for $c>0$, $y_0\in \RR$,}$$ where $Q_c(x)=c^{\frac 13} Q(\sqrt{c} x)$ and    $Q$ satisfies
$$Q'' + Q^4 = Q, \quad 
Q(x)=\Bigg(\frac 5 {2 \cosh^2\big(\frac 3 2 x\big)} \Bigg)^{\frac 13}
.$$

We call \emph{outgoing  multi-soliton}   a solution $u(t)$ of \eqref{kdv} such that
\begin{equation}\label{multisol}
	\lim_{t\to +\infty} \bigg\|u(t)- \sum_{j=1}^NQ_{c_j} (. - c_j t-\Delta_j)\bigg\|_{H^1} =0,
\end{equation}
for some $N\geq 2$, $0<c_N<\ldots<c_1$, 
and $\Delta_j \in \RR$.
For a given set of such parameters, the existence and uniqueness of an outgoing  multi-soliton was proved in \cite{Ma2} (see   also \cite{MMT} for previous related results), together with the following regularity and    convergence properties:
$u(t)\in \cap_{s\geq 1} H^s$, and for some  $\gamma>0$, for all $s\geq 0$,
$$ \hbox{for all $t>0$,}\quad 
 \bigg\|u(t)- \sum_{j=1}^NQ_{c_j} (. - c_j t-\Delta_j)\bigg\|_{H^s}\leq C_s e^{-\gamma t}.
$$
Similarly, we call \emph{ingoing  multi-soliton}, a solution $u(t)$ of \eqref{kdv} such that
\begin{equation}\label{multisolm}
	\lim_{t\to -\infty} \bigg\|u(t)- \sum_{j=1}^NQ_{c_j} (. - c_j t-\Delta_j)\bigg\|_{H^1} =0.
\end{equation}
We call \emph{pure  multi-soliton}, a solution of \eqref{kdv} which is  both an  ingoing and an outgoing multi-soliton, 
possibly with   different numbers of solitons $N^\pm$ and  different speeds and position parameters $c_j^\pm$, $\Delta_j^\pm$ as $t\to +\infty$ or as $t\to -\infty$.
The aim of this paper is to investigate   the relation between ingoing and outgoing multi-solitons of the nonintegrable quartic (gKdV) equation \eqref{kdv}, and more precisely to prove the nonexistence of pure multi-solitons for an explicit range of speeds.

\smallskip

It is well-known that  for the  (KdV) and (mKdV) equations, i.e. in the integrable cases,
\begin{align}\label{eq:int}
  \partial_t u + \partial_x (\partial_x^2 u   + u^2) =0 \quad &\hbox{ (KdV) }\\  \partial_t u + \partial_x (\partial_x^2 u   + u^3) =0 \quad & \hbox{ (mKdV)} \label{eq:int2}
 \end{align}
this question was completely settled by integrability (see e.g. \cite{FPU, Z, KZ, LAX1, HIROTA, WT, Miura}).
Indeed,    there exist explicit pure multi-solitons for any   parameters  and they  
  are the only ingoing   multi-solitons.
In particular, the collision of any number of solitons is always elastic, meaning that neither the number of solitons, nor their speeds, are   changed by the collision (the trajectories of the solitons are in general  shifted).
 We refer to \cite{Miura} and references therein for a review of results for the integrable models \eqref{eq:int}, \eqref{eq:int2}.

\smallskip

For nonintegrable models, existence and properties of multi-solitons has also become a classical question, studied through different points of view (see e.g. \cite{Miura, SHIH, MMT, Mi, CGHHS, HHGY}).
For  the quartic (gKdV) equation, the authors of the present paper have already adressed this question in the case of $2$-solitons with speeds   $0<c_2<c_1$ in the   following  two cases, for $\epsilon>0$ small:
\begin{itemize}
\item[{(a)}] Solitons of different speeds: $\frac {c_2}{c_1}\leq \epsilon$. See \cite{MMcol1}
\item[{(b)}] Solitons with almost equal speeds: $0<1-\frac {c_2}{c_1}\leq \epsilon$. See \cite{MMalmost}.
\end{itemize}
In \cite{MMcol1} and \cite{MMalmost}, under   condition (a) or (b), we have given a refined description of ingoing $2$-solitons  for all $t$ and $x$, up to some order of $\epsilon$. From this description, we could deduce the  following facts.
(1) The $2$-soliton structure is globally stable  in time in $H^1$, in the sense that an ingoing $2$-soliton is  for all time   the sum of two solitons at the main order.
(2)  Ingoing $2$-solitons \emph{cannot} be outgoing $2$-solitons.
In particular, no pure $2$-soliton can exist in these two regimes. In contrast with the integrable cases, the collision is \emph{inelastic}. From explicit computations, we could find lower bounds and upper bounds on the size of the residual term due to the collision.

\smallskip

Summarizing, ingoing $2$-solitons are   well-understood for all time under assumptions (a) and (b) for $\epsilon>0$ small enough. However, \emph{the value of $\epsilon$ for which the results in \cite{MMcol1} and \cite{MMalmost} hold is not explicit}
because of the complexity of the computations and the  perturbative nature of the proofs. Another restriction concerns the number of solitons. In \cite{MMcol1} and in \cite{MMalmost}, the proofs are only written   for $2$-solitons and it  would be quite  involved to extend them  to   $N$-solitons.

\smallskip

In view of the inelasticity  results in \cite{MMcol1} and \cite{MMalmost}, we conjecture that for all $0<c_2<c_1$, the corresponding ingoing $2$-solitons of the quartic (gKdV) equation are not  pure $2$-solitons. In other words, there should not exist pure $2$-solitons.
In fact, we expect  that such property is true for general nonintegrable systems.  Perelman's work \cite{Perelman} for the nonlinear Schr\"odinger equation,   Munoz' works \cite{Mukdv,Munoz, Munoz2}, and \cite{MMM} for the BBM equation, are  other evidences of such belief.

\smallskip

In this paper, we attack the problem through a different strategy.
The main point is to prove nonexistence of   pure 2-soliton of the quartic (gKdV) equation
\emph{without trying to describe the solution for all time} and  for an {\it explicit} range of speeds.
Here, the approach is \emph{not perturbative}, and we do not need to compute the main order of the solution for all $t,x$. Indeed, a contradiction is obtained by estimating only the tail of the solution $u(t,x)$ for large $t$ and   large $x$. The knowledge of the solution on compact sets of space-time is not required. 
Moreover, the method allows to consider the   case of $N$-solitons, for any $N\geq 2$, without significant changes, in contrast with \cite{MMcol1} and \cite{MMalmost}.
Note that we   consider  the quartic (gKdV) equation because it is a typical nonintegrable system, relatively simple and not perturbative of the integrable cases (see \cite{Tao}, \cite{KM}), but we expect our approach to be general and flexible enough to extend to other models.

\subsection{Statement of the result}


 The following is the main result of this paper.
 
\begin{theorem}\label{th:1}
Let $N\geq 2$. Let $u(t)$ be an outgoing multi-soliton of \eqref{kdv}, with parameters $\Delta_1$,\ldots, $\Delta_N\in \RR$, $0<c_N<\ldots<c_1=1$,
\begin{equation}\label{multisolqua}
	\lim_{t\to +\infty} \bigg\|u(t)-  \sum_{j=1}^NQ_{c_j} (. - c_j t-\Delta_j )\bigg\|_{H^1} =0.
\end{equation}
Assume that
\be\label{speeds}
 \sum_{j=2}^N  \left(1 - c_j\right)^2    < \frac {1}{16} .\ee
Then, $u(t)$ is not an ingoing multi-soliton at $-\infty$.

In particular, under assumption \eqref{speeds}, there exists no pure multi-soliton of   \eqref{kdv} with speeds $1,c_2,\ldots,c_N$ at $+\infty$ or at $-\infty$.
\end{theorem}
 
\begin{remark}
In the case $N=2$, Theorem \ref{th:1} proves the nonexistence of pure $2$-soliton under the   condition
$\frac 34 < \frac {c_2}{c_1} <1$.
\end{remark}
 
The strategy  of the proof is   different from the one  in \cite{MMcol1} and \cite{MMalmost}
where the goal was to describe ingoing $2$-solitons for all $t,x\in \RR$, by a pertubative analysis.
In this paper, to prove nonexistence of multi-solitons,  we do not need to understand the solution on bounded sets of $(t,x)$ and we only consider the tail of the solution as $|x|\sim +\infty$. This approach involves different computations which we can perfom for an explicit range of speeds and for any number of solitons.

\medskip

For the sake of contradiction, we assume the existence of a solution $u(t)$ of \eqref{kdv} which is both an outgoing multisoliton ($t\to +\infty$) with parameters $N\geq 2$, $0<c_N<\ldots<c_1$, $\Delta_1, \ldots,\Delta_N\in \RR$:
\begin{equation}\label{msm}
	\lim_{t\to +\infty} \bigg\|u(t)- \sum_{j=1}^NQ_{c_j} (. - c_j t-\Delta_j)\bigg\|_{H^1} =0,
\end{equation}
and an ingoing multisoliton  ($t\to -\infty$)  with parameters $N^-\geq 2$, $0<c_{N^-}^-<\ldots<c_1^-$, $\Delta_1^-, \ldots,\Delta_{N^-}^-\in \RR$:
\begin{equation}\label{msp}
	\lim_{t\to -\infty} \bigg\|u(t)- \sum_{j=1}^{N^-}Q_{c_j^-} (. - c_j^- t-\Delta_j^-)\bigg\|_{H^1} =0.
\end{equation}
We assume 
\begin{equation}
c_1=1
\end{equation}
and
\be\label{speeds2}
 \sum_{j=2}^N  \left(1 -  {c_j} \right)^2    \leq  \frac 1{16}.
 \ee 
 \medskip

The contradiction   comes from the following steps:

\medbreak

\noindent{(a) \emph{Control of the speeds at $-\infty$.}} From the three conservation  laws (mass, energy and integral) and elementary algebraic arguments, we claim that the speeds at $-\infty$ are  also close to $1$ in the following sense (see Section 2)
\begin{lemma}\label{sp}
Assume \eqref{speeds}, \eqref{msm}, \eqref{msp}.
For all $j=1,\ldots,N^-$, 
\be\label{sp1}
\frac {16}{25} < c_j^- < \frac 32.
\ee
Moreover,
\be\left| N^- - N\right| \leq \frac {\sqrt{N}}{8}.\ee
\end{lemma}

\medskip

\noindent (b) \emph{Decay on the right for positive time.} From the behavior of the solution $u(t)$ as $t\to -\infty$ in the energy space \eqref{msp}, the lower bound $c_N^-> \frac {16}{25}$ and usual monotonicity arguments, we claim that the solution $u(t)$ satisfies  exponential decay property on the right of the   soliton $Q(x-t-\Delta_1)$.

\smallskip 

Let  $j_0\in \{1,\ldots,N-1\}$ be such that
\be\label{defjz}
 \sigma_0:= \min_j \sqrt{c_{j+1}}(c_j-c_{j+1})=\sqrt{c_{j_0+1}}(c_{j_0}-c_{j_0+1}),
\ee
and set
\be\label{defgamma}
\gamma_0=   \sqrt{ c_{j_0}  - \frac 34  {c_{j_0+1}} } - \frac 12 \sqrt{ c_{j_0+1}},
\quad 
x_0(t) = \left( \frac {\sigma_0}{\gamma_0} + c_{j_0}\right) t - K_0,
\ee
where $K_0>1$ is a large constant to be fixed later. Note by \eqref{clbm1} that 
$
\frac {\sigma_0}{\gamma_0}  + c_{j_0} > 2 c_{j_0+1}> \frac 32.
$
We claim (see Section \ref{sec:3})
\begin{lemma}\label{le:droiteN}
 There exists $t_0(K_0)>0$ such that,
 \begin{equation}\label{infty2}
\forall t>t_0(K_0),\quad |u(t,x_0(t))  | \leq  
  e^{- 2 \sigma_0 t}.
\end{equation}
\end{lemma}

 \medbreak
 
\noindent (c) \emph{Approximate solution and lower bound.} 
We establish the following result, which is the main new ingredient of the paper.
\begin{proposition}\label{prop:1N}
Assume \eqref{speeds2}.  
There exist $  C_1>0$ independent of $K_0$ and $ t_1(K_0)>0$ such that, for $K_0$ large enough,
\begin{equation}\label{eqpr:1N}
\forall t\geq t_1(K_0),\quad 
|u(t,x_0(t)) | \geq  C_1 e^{\gamma_0 K_0} e^{-2 \sigma_0 t}  .
\end{equation}
\end{proposition}

Fix $K_0>0$ such that $C_1 e^{\gamma_0 K_0} >2$.
Combining Lemma \ref{le:droiteN} and Proposition \ref{prop:1N},  we obtain a contradiction for  $t>\max(t_0(K_0),t_1(K_0))$.

\medskip

Let us sketch the proof of Proposition \ref{prop:1N}.
The key point is to construct  an explicit approximate solution $V(t)$ of the problem as $t\to +\infty$ (see Section \ref{sec:new4}). We briefly sketch the construction of $V$  in the 2-soliton case, i.e. for $N=2$.  Let
$$V =R_1 +R_2  +Z, \quad R_1(t,x)=Q(x-t-\Delta_1), \ R_2(t,x)=Q_{c_2}(x-c_2t-\Delta_2),
$$
where 
$$
Z_t + (Z_{xx} + 4 R_1^3 Z  )_x \approx - 4 (R_1^3 R_2)_x.
$$
In the equation of $Z$ above,  we focus on the main interaction term (see Section~\ref{sec:new4} for the control of all error terms).
For this term,  we replace $R_{2}$ by its asymptotics for $x-c_2t \gg 1$:
$$  R_1^3 R_2 \approx  10^{\frac 13}   c_2^{\frac 13} e^{-\sqrt{c_2} (x-c_2t-\Delta_2)} R_1^3 =
c_0e^{-\sqrt{c_2}(1-c_2)t}e^{-\sqrt{c_2}(x-t -\Delta_1)} R_1^3,$$
where $c_0=10^{\frac 13}c_2^{\frac 13} e^{-\sqrt{c_2}(\Delta_1{-}\Delta_2)}$. An explicit solution 
of 
$$
Z_t + (Z_{xx} + 4 R_1^3 Z  )_x = - 4 c_0  \left(e^{-\sqrt{c_2}(1-c_2)t}e^{-\sqrt{c_2}(x-t -\Delta_1)} R_1^3\right)_x$$
is
$Z(t,x) = 4 c_0 e^{-\sqrt{c_2}  (1-c_2)t} A(x{-}t{-}\Delta_1)$ where   $A(x)$   satisfies the following ODE
\be\label{ode}
(- A'' + A - 4 Q^3 A)' + \sqrt{c_2}(1-c_2) A =  (e^{-\sqrt{c_2} x} Q^3)'.
\ee
Moreover,  for $\frac 34 <c_2<1$, we prove the following asymptotic property  
\be\label{generic}A(x) \mathop\sim_{x\to  +\infty} a e^{-\gamma_0 x}\quad \hbox{where
$a\neq 0$ and $\gamma_0 =  \sqrt{1  - \frac 34  {c_2}} - \frac 12\sqrt{c_2}$ }.\ee
 The proof of   property \eqref{generic} requires more than  standard ODE techniques, and involves  Virial type arguments, introduced in \cite{MM1}, \cite{yvanSIAM} and \cite{MMas1} to study the flow of the evolution problem \eqref{kdv} near   solitons.

 It follows from \eqref{generic} that $V$ 
  satisfies  the following lower bound, for $\kappa>0$,
  $t$ and $x$ large enough,
\be\label{lowV}
|V(t,x )|\geq \kappa   e^{-\gamma_0 (x-t)}  e^{-\sqrt{c_2} (1-c_2)t}.
\ee
Such an approximate solution $V(t)$ being constructed (the actual approximate solution is more refined), by usual techniques (\cite{MMT}, \cite{Ma2}), we compare the solution $u(t)$ with $V(t)$, for $t$ large: 
$$\|u(t)-V(t)\|_{H^1} \leq C e^{-2 \sqrt{c_2} (1-c_2)t}$$
and we  obtain the desired lower bound on $u(t,x)$ at $x=x_0(t)$ (see Section~\ref{sec:4}).

\medbreak

\noindent{\bf Comment on assumption \eqref{speeds}.} The assumption on the speeds $c_1,\ldots,c_N$ in  \eqref{speeds} is   not optimal, and we even conjecture that the result holds for any choice of speeds.
However,   we believe that to obtain the more general result will require much harder analyis.
Even  considering the simplest case of a $2$-soliton with speeds $c_1=1$ and $c_2$, we see several 
difficulties to extend the nonexistence result to any $0<c_2<1$.
\begin{enumerate}[1)]
\item For $0<c_2\le 1/3$, the method outlined above does not work direclty for algebraic reasons. Indeed, 
$0<c_2\le 1/3$ implies $\gamma_0 \geq \sqrt{c_2}$, and thus  the approximate solution has the same decay as   $R_2$, and no direct contradiction can follow from such a lower bound.
This  is related to the fact that  the proof of  inelasticity for $c_2$ close to $0$  in \cite{MMcol1} requires a higher order expansion than the one in \cite{MMalmost} for $c_2$ close to~$1$.

\item For all $\frac 13 <c_2<1$, we expect that the function $A(x)$ defined above has the generic decay 
\eqref{generic}, which is essential in our proof, but we were able to prove this fact only for $c_2\in [c_0,1]$, where $c_0<\frac 34$ is close to $\frac 34$.

\item The restriction \eqref{speeds} on the values of $c_2$ also comes from the proof of the decay property obtained in Lemma \ref{le:droiteN}. We prove Lemma \ref{le:droiteN} by known and simple energy localization arguments, which are clearly not optimal. Replacing these arguments by a sharper asymptotic analysis would certainly improve the range of admissible  $c_2$, but   without  approching the special value $c_2=\frac 13$.

\end{enumerate}

\noindent In the case of $2$-solitons, it is proved in  Proposition \ref{th:UNIQ} that $N^-=2$ and 
$c_1^-=1$, $c_2^-=c_2$ without condition on $c_2$. For $N\geq 3$, it is not clear how to prove such a rigidity property, or even how to obtain a lower bound such as \eqref{sp1} without a strong   assumption on the speeds at $+\infty$ such as \eqref{speeds}. In particular, we cannot replace \eqref{speeds} by  $\frac 34 < c_j < 1$, $\forall j=2,\ldots, N$.


\medskip

\noindent{\bf Acknowledgement.}  
This work is   partly supported by the project ERC 291214 BLOWDISOL.


\section{Rigidity of multi-soliton parameters}
In this section, using the three conservation laws \eqref{mass}, \eqref{energy} and \eqref{int},  we prove Lemma \ref{sp}, which controls the speeds at $-\infty$ for an outgoing multi-soliton under assumption \eqref{speeds}. We also state and prove an independent unconditional result of rigidity of the speeds at $\pm \infty$ for a $2$-soliton.

Note that the arguments can be extended to other power nonlinearities.

\subsection{Conservation laws on ingoing $N$-solitons}
We first claim the following result to be proved in Appendix \ref{ap:B}.

\begin{lemma}\label{lun}
Let $N\geq 2$, $0<c_N<\ldots<c_1$ and $\Delta_1, \ldots,\Delta_N\in \RR$.
Let $u(t)$ be the solution of \eqref{kdv} satisfying
\begin{equation}\label{pinfdeux}
\lim_{t\to +\infty} \bigg\|u(t)- \sum_{j=1}^N Q_{c_j} (.- c_j t  -\Delta_j)\bigg\|_{H^1} =0.
\end{equation}
Then, for all $t$,
$$
\int u^2(t) = \sum_{j=1}^N \int Q_{c_j}^2= \left(\sum_{j=1}^N  c_j^{\frac 1{6}}\right) \int Q^2 .
$$$$
E(u(t)) = \sum_{j=1}^N E(Q_{c_j})= \left(\sum_{j=1}^N  c_j^{\frac 7{6}}\right) E(Q).
$$
Moreover, $u(t)\in L^1$ and 
$$
\int u(t) = \sum_{j=1}^N \int Q_{c_j}= \left(\sum_{j=1}^N  c_j^{-\frac 1{6}}\right) \int Q .
$$
\end{lemma}

\subsection{Proof of Lemma \ref{sp}}\label{sec:2.1}
 
Using Lemma \ref{lun} and \eqref{msm}-\eqref{msp}, the following identities hold
\be\label{id}
\sum_{j=1}^N  c_j^{\frac 7{6}} = \sum_{j=1}^{N^-}  (c_j^-)^{\frac 7{6}},\quad 
\sum_{j=1}^N  c_j^{-\frac 1{6}} = \sum_{j=1}^{N^-} (c_j^-)^{-\frac 1{6}},\quad 
\sum_{j=1}^N  c_j^{\frac 1{6}} = \sum_{j=1}^{N^-}  (c_j^-)^{\frac 1{6}}.
\ee
Consider the function $f(x)$ for $x>0$ defined as
$$
f(x) = x^7 +\frac 3 x - 4x\quad \hbox{so that}\quad f'(x)=7 x^6 - \frac 3{x^2} - 4,\quad
f''(x) = 42 x^5 + \frac {6}{x^3}>0.
$$
In particular, $f(1)=f'(1)=0$ and elementary computations show that
$$
\forall x>0,\quad f''(x) \geq f''\left( (3/35)^{\frac 18}\right) = \frac {48}{5} \left(\frac {35}3\right)^{\frac 3{8}}:=2 m_1,
$$
$$
 \frac 34 \leq  x \leq 1 \ \Rightarrow \ 
f''(x) \leq f''(1)=48 .
$$
We deduce:
\be\label{nf}  \forall x>0,\ f(x) \geq m_1 (1-x)^2;\quad
\forall x\in [\tfrac 34 ,1],\
 f(x)\leq 24 (1-x)^2.
\ee
Combining the identities in \eqref{id}, we have
$$
 \sum_{j=1}^{N^-} f\left((c_j^-)^{\frac 16}\right) =\sum_{j=1}^N f\left(c_j^{\frac 16}\right)  .
$$
Using \eqref{nf}, \eqref{speeds} and $1-c_j^{\frac 16} \leq 4 (1-(\frac 34)^{\frac 16}) (1-c_j)$ (since $\frac 34 <c_j<1$),
\begin{align*}
&\left|1-(c_j^-)^{\frac 16}\right|^2 \leq 
\sum_{j=1}^{N^-} \left|1-(c_j^-)^{\frac 16}\right|^2   \leq \frac {24}{m_1}  \sum_{j=2}^N\left|1-c_j^{\frac 16}\right|^2 \\
& \leq 80 \left(\frac 3 {35}\right)^{\frac 38} \left( 1-\left(\frac 34\right)^{\frac 16}\right)^2    \sum_{j=2}^N  \left(1 -   {c_j} \right)^2   
\leq  5\left(\frac 3 {35}\right)^{\frac 38} \left( 1-\left(\frac 34\right)^{\frac 16}\right)^2 :=m_2 .
\end{align*}
Thus, for all $j=1,\ldots,N^-$, by elementary computations,
$$
\frac {16}{25} <  \left(1 -  {\sqrt{m_2 }} \right)^{6} \leq c_j^-\leq \left(1+  {\sqrt{ m_2 }}  \right)^{6}  <\frac 32 .
$$

Now, we prove the control of $N^-$. Indeed, we have
\begin{align*}
& \left| \sqrt{N} - \left( \sum_{j=1}^N c_j^{\frac 16}\right)^{\frac 12}\right|  =\left| \left( \sum_{j=2}^N 1^2\right)^{\frac 12}- \left( \sum_{j=1}^N \left|c_j^{\frac 1{12}}\right|^2\right)^{\frac 12}\right|   \leq 
\left( \sum_{j=2}^{N} \left|1-c_j^{\frac 1{12}}\right|^2 \right)^{\frac 12} 
\\& \leq \left(1+\left(\tfrac 34\right)^{\frac 1{12}}\right)^{-1} \left( \sum_{j=1}^{N} \left|1-c_j^{\frac 1{6}}\right|^2 \right)^{\frac 12}\leq 
\left(1+\left(\tfrac 34\right)^{\frac 1{12}}\right)^{-1} \left( 1-\left(\frac 34\right)^{\frac 16}\right) :=a_1,
\end{align*}
\begin{align*}
 &\left| \sqrt{N^-} - \left( \sum_{j=1}^{N^-} (c_j^-)^{\frac 16}\right)^{\frac 12}\right|   =\left| \left( \sum_{j=1}^{N^-} 1\right)^{\frac 12}- \left( \sum_{j=1}^{N^-} (c_j^-)^{\frac 16}\right)^{\frac 12}\right|   \leq 
\left( \sum_{j=1}^{N^-} \left|1-(c_j^-)^{\frac 1{12}}\right|^2 \right)^{\frac 12} 
\\& \leq \left(1+\left(\tfrac 45\right)^{\frac 1{6}}\right)^{-1} \left( \sum_{j=1}^{N^-} \left|1-(c_j^-)^{\frac 1{6}}\right|^2 \right)^{\frac 12}\leq 
 \left(1+\left(\tfrac 45\right)^{\frac 1{6}}\right)^{-1}5^{\frac 12}\left(\frac 3 {35}\right)^{\frac 3{16}} \left( 1-\left(\frac 34\right)^{\frac 16}\right) := a_2 .
\end{align*}
Thus, by \eqref{id},
$$
\left| \sqrt{N}- \sqrt{N^-}\right| \leq a_1+a_2:=a$$
and so by explicit computations,
$$
\left| N^- - N\right| \leq 2 a \sqrt{N} + a^2 \leq \left(2a + \frac {a^2}{\sqrt{2}}\right)\sqrt{N}
< \frac {\sqrt{N}}{8}.
$$
In particular, if $N\leq 64$, then $N^-=N$. 

\subsection{Rigidity result for two solitons}\label{sec:2.2}
 
 In the case of an ingoing $2$-soliton, we prove an unconditional result.
 For any $0<c<1$, we claim that if the ingoing $2$-soliton  is also an outgoing $N$-soliton, then  $N=2$ and the speeds at $+\infty$ and $-\infty$ are the same.
In particular, it is a symmetric $2$-soliton. This result is not needed for the proof of Theorem \ref{th:1} but it is proved for its own interest.  Such question remains open for $N\geq 3$. 
 
\begin{proposition}[Rigidity of  $2$-solitons]\label{th:UNIQ}
Let $0<c<1$.
Let $u(t)$ be the outgoing $2$-soliton   of \eqref{kdv} satisfying
\begin{equation}\label{minf}
\lim_{t\to +\infty} \|u(t)- Q (.- t ) -Q_{c} (. - c t )\|_{H^1} = 0.
\end{equation} 
Assume that $u(t)$ is   an ingoing  multi-soliton, i.e. there exist
$0<c_N<\ldots<c_1$, and $\Delta_1, \ldots,\Delta_N$ such that
\begin{equation}\label{pinf}
\lim_{t\to -\infty} \bigg\|u(t)- \sum_{j=1}^N Q_{c_j} (.- c_j t  -\Delta_j)\bigg\|_{H^1} =0.
\end{equation}
Then,
\begin{itemize}
\item[{\rm (i)}] $u(t)$ is a pure $2$-soliton,
$$N=2 \quad \hbox{and } \quad c_1=1,\quad c_2=c.$$
\item[{\rm (ii)}] 
There exist $T_0, \ Y_0\in \RR$ such that
\begin{equation}\label{eq:sym}
 u(t,x)=u(-t+T_0,-x+Y_0).
\end{equation}
\end{itemize}
\end{proposition}
Let $u(t)$ be a solution of \eqref{kdv} as in the statement of Proposition \ref{th:UNIQ}.
Property~(ii) is a direct consequence of (i) and the uniqueness result in \cite{Ma2}.

\smallskip

We now prove (i).
By Lemma~\ref{lun}, we have  
$$
  1+ c^{\frac 1{6}}= \sum_{j=1}^N  c_j^{\frac 1{6}},\quad 
 1+ c^{\frac 7 {6}} =  \sum_{j=1}^N  c_j^{\frac 1{6}} ,\quad
  1+ c^{-\frac 1{6}}  =  \sum_{j=1}^N  c_j^{-\frac 1{6}}.
$$
Setting $a_j = c_j^{\frac 1{6}}$ and $x=c^{\frac 1{6}}$,
Proposition \ref{th:UNIQ} (i)   follows from the following elementary result.
\begin{lemma}
Let $0<x<1$, $N\geq 2$   and $0<a_N<\ldots<a_1$ be such that
\be\label{ftrois} \sum_{j=1}^N a_ j = 1+x, \quad 
 \sum_{j=1}^N a_j^7  = 1 + x^7, \quad
 \sum_{j=1}^N \frac 1{a_j} = 1+\frac 1x.
\ee
Then, $N=2$, $a_1=1$ and $a_2=x$.
\end{lemma}
 
\begin{proof}
The case $N=2$ is easily treated. Let $a_1=a$ and $a_2=b$, $0<b<a$ be such that 
\begin{align*}
& a+b = 1+x,\\
&  a^7 +   b^7 = 1 + x^7 = 1+(a+b-1)^7.
\end{align*}
Of course, $a=1$, $b=x$ is a solution.
For $0<a<1+x$, $a\neq 1$, set $f(b) = a^7 +   b^7-  1-(a+b-1)^7$.
We see that $f(1)=0$. Moreover, $f'(b) = 7(b^6-(a+b-1)^6)$, and thus   $f'(b)$ has no zero
on $(0,1)$. It follows that $f$ has no zero on $[0,1)$ and so there are no other solution than $a=1$, $b=x$ for $N=2$.
 
\medskip

We now consider the case $N\geq 3$.
We define the bounded set
$$
\Omega=\Bigg\{(a_1,\ldots,a_N)\in (\RR_+^*)^N \ | \ \sum_{j=1}^N a_ j = 1+x,\  \sum_{j=1}^N a_j^7  = 1 + x^7\Bigg\}.
$$
and we look for the minimum on $\Omega$  of the following positive function $F$:
$$
F(a_1,\ldots,a_N) = \sum_{j=1}^N \frac 1{a_j}. 
$$
Since $\lim_{a_j\to 0^+} F(a_1,\ldots,a_N)=+\infty$, $F$ reaches it minimum on $\Omega$.

Note that if  $(a_1,\ldots,a_N)$ is a point of $\Omega$ where the gradients of the functions $\sum_{j=1}^N a_ j$ and $\sum_{j=1}^N a_j^7$ are colinear,then $a_j=a$ for all $j\in \{1,\ldots,N\}$. Thus,
$$
  N a  = 1+x,\quad 
  N a^7  = 1 + x^7.
$$
It follows that $\frac 1 N< a < \frac 2 N$ and   $a^6 = \frac {1+x^7}{1+x}>\frac 12$ which imply 
$(\frac 2N)^{6} \geq \frac 12$ and so $N\leq 2^{1+\frac 16}$. This is a contradiction and so no such point exists on $\Omega$.

Therefore, we can apply the method of Lagrange multipliers to characterize extrema of $F$ on $\Omega$.
For a critical point $(a_1,\ldots,a_N)\in \Omega$ of $F$, there exist $\lambda,\mu\in \RR$ such that
$$
\forall j=1,\ldots, N,\quad \frac 1{a_j^2}  = \lambda + \mu a_j^6.
$$
Let $g(\alpha) = \mu \alpha^4 + \lambda \alpha - 1$.
We see that $g'$ has at most one root on $[0,+\infty)$ and $g$ has at most two roots on $[0,+\infty)$.
We have already observe that $\Omega$ contains no point of the form $(a,\ldots,a)$.

Therefore, the $(a_j)$   take exactly two different values:  there exist $0<b<a$ and  $1\leq k < N$ such that
\begin{align*}
& k a +(N-k) b = 1+x,\\
& k a^7 + (N-k) b^7 = 1 + x^7 .
\end{align*}

$\bullet$ For   $1/3\leq x <1$ :
Note that for all $y>0$, $\frac 1 y  \geq 2-y$. Thus,
$$
\frac k a + \frac {N-k} b  \geq 2 N - k a - (N-k) b.
$$
Since
$$
ka + (N-k) b = 1+x,
$$
we find
$$
\frac k a + \frac {N-k} b  \geq 2N - 1- x > 4\geq  1+ \frac 1x.
$$
It follows that at such a critical point, $F$ is strictly greater that $1+\frac 1x$.

\medskip

$\bullet$  For  $k=1$, $0<x<1/3$.
Then,
\begin{align*}
& a -1 +(N-1) b = x,\\
& a^7-1 + (N-1) b^7 = x^7 .
\end{align*}
Since
$$
|a^7 - 1| = |a-1| (a^6+a^5+a^4+a^3+a^2+a+1)\geq  |a-1|,
$$
we obtain from the second identity
$$
|a-1|\leq x^7 + (N-1) b^7.
$$
Combining this with the first identity and then using $b<\frac 23$ (since $3 b<a+2b\leq 1+x<2$),  we get
$$
(N-1) b < x + x^7 + (N-1) b^7 < \frac {11}{10} x + \frac {N-1}{10} b,
$$
and so $\frac 9{11} (N-1)b < x$. In particular
$$
 \frac {N-1}{b} \geq \frac 9{11} \frac {(N-1)^2}{x} \geq \frac 9{11} \frac 4 x> 1 + \frac 1 x.
$$
Again, at such a critical point, $F$ is strictly greater that $1+\frac 1x$.

\medskip

$\bullet$  For  $2\leq k\leq N-1$, $0<x<1/3$.
Since $0<b<a$ and
$$
Nb <ka + (N-k) b = 1+x,
$$
we have
$$
a \leq \frac {1+x} k,\quad b \leq \frac {1+x} N.
$$
Thus,
$$
1+x^7= ka^7 + (N-k) b^7 \leq \left( k^{-6} + N^{-6} \right) (1+x)^7
\leq  \left( 2^{-6} + 3^{-6} \right) \left( \frac 43\right)^7< 1,
$$
a contradiction. This means that no such critical point exist in this case.
\end{proof}

\section{Pointwise decay estimates for ingoing multisoliton}\label{sec:3}

This section is devoted to the proof of Lemma \ref{le:droiteN} by standard monotonicity arguments (see e.g. \cite{MM1}, \cite{MMnonlinearity} and \cite{Kato}). 

\subsection{Monotonicity result}
Set
$$
\phi(x) = \frac 2 {\pi} \arctan(\exp\,  x),\quad
\phi'(x) = \frac 1{\pi \cosh x},
$$
so that 
\be\label{phi3}
\phi''' \leq \phi' \hbox{ on $\RR$.}
\ee
Let us recall the following result (see \cite{MMnonlinearity}) whose proof is given in Appendix \ref{ap:B} for the sake of completeness.
\begin{lemma}[Mass-energy monotonicity]\label{le:mono}
Let   $0<\sigma<\sigma'<c_0$ and $C_0>0$.
There exists $\alpha_0>0$ such that the following holds.
Let $u(t)$ be a solution of \eqref{kdv}  such that there exists $R>1$ with
\begin{align}
\forall  {t\in [t_1,t_2]},\quad    \|u(t)\|_{H^1} \leq C_0,  \quad \|u(t)\|_{L^2(x>c_0 t +R)} \leq \alpha_0. \label{eq:mono3}
\end{align}
Then, there exists $C=C(\sigma,\sigma',c_0,C_0,R)>0$ such that,   for all   $x_0>0$,
\begin{align}
& \int (u_x^2+u^2  )(t_2,x) \phi\left(\sqrt{\sigma} (x-c_0 t_2-x_0) \right)dx
\nonumber \\& \leq
2 \int \big( u_x^2 + u^2   \big)(t_1,x)\, \phi\big( \sqrt{\sigma} (x{-}c_0 t_1{-} (c_0{-}\sigma') (t_2{-}t_1) {-} x_0)\big) dx+ C e^{-\sqrt{\sigma} x_0}.\label{eq:mono4}
\end{align}\end{lemma}

\subsection{Decay on the right. Proof of Lemma \ref{le:droiteN}}\label{sec:3.2}

\noindent {\it Step 1.} 
Decay  at $t=0$.  Let $0<\sigma_1<c_N^-$. We claim that there exists $C>0$ such that
\be\label{oldle13}
\forall x_0>0,\quad  \int  \left( u_x^2+u^2 \right) (0,x)  \phi(\sqrt{\sigma_1} (x-x_0) )dx  \leq C e^{- \sqrt{\sigma_1} \, x_0}.
\ee

Let $\sigma=\sigma_1$, $ \sigma' = \frac 12 (\sigma_1+c_N^-)$, $c_0=c_N^-$ and $C_0=\sup_t\|u(t)\|_{H^1}$. Let $\alpha_0$ be given by Lemma \ref{le:mono}. 
From \eqref{msp}, for $t_0>0$ large enough,  for all $t<-t_0$, for all $x$,
\begin{align*}
|u(t,x)|& \leq \sum_j |Q_{c_j}(x-c_jt-\Delta_j)|    + \left\|u(t)-\sum_j  Q_{c_j}(.-c_jt-\Delta_j) \right\|_{H^1}\\
& \leq C \sum_{j=1}^N e^{-\sqrt{c_j} (x-c_j t)} + \frac 12 {\alpha_0}.
\end{align*}
Thus, there exists $R>0$ such that
$$
\sup_{t\leq -t_0} \|u(t)\|_{L^2(x>   c_N^- t + R)}< \alpha_0.
$$
By possibly taking a larger $R$, we also have
$$
\sup_{t\leq 0} \|u(t)\|_{L^2(x>   c_N^- t + R)}< \alpha_0.
$$
Applying Lemma \ref{le:mono}  on $[t,0]$, for any $t<0$, for all $x_0>0$,
\begin{align*}
&\int  \left( u_x^2 + u^2\right) (0,x)  \phi(\sqrt{\sigma_1} (x-x_0)) dx \\& \leq 
2 \int \left(u_x^2+u^2\right)(t,x) \phi(\sqrt{\sigma_1} (x-x_0- \sigma' t) )dx + C e^{-\sqrt{\sigma_1 }x_0}.
\end{align*}
By \eqref{msp} and the definition of $\phi$,
$$\lim_{t\to -\infty}\int \left(u_x^2+u^2\right)(t,x)  \phi(\sqrt{\sigma_1} (x-x_0- \sigma' t) ) dx =0,$$
and \eqref{oldle13} follows.

\medskip

\noindent{\it Step 2.} Decay on the right for   $t>0$.
Let $\sigma_1 \leq \sigma_2<\sigma_3<1$. We claim
 \begin{equation}\label{infty2vrai}
 \forall x>0, \forall t>0,\quad 
|u(t,t+x)  | \leq  
 C \left( e^{-\frac 12 \sqrt{\sigma_1} \,(x+(1-\sigma_3) t)}+  e^{- \frac 12 \sqrt{\sigma_2}\, x}\right).
\end{equation}

Let $\sigma=\sigma_2$, $\sigma'=\sigma_3$,  $c_0=1$ and $C_0=\sup_t\|u(t)\|_{H^1}$. Let $\alpha_0$ be given by Lemma~\ref{le:mono}. 
As before, from \eqref{msp}, there exists $R>0$ such that
$$
\sup_{t\geq 0}\|u(t)\|_{L^2(x\geq  t+R)}< \alpha_0.
$$
Applying Lemma~\ref{le:mono},  for $t>0$, $x_0>0$,
\begin{align*}
& \int  \left( u_x^2 + u^2\right) (t,x) \phi(\sqrt{\sigma_2} (x-t-x_0)) dx \\ &\leq 
2 \int \left( u_x^2 + u^2\right) (0,x) \phi(\sqrt{\sigma_2}(x-(1-\sigma_3)t-x_0))dx + C e^{-\sqrt{\sigma_2} x_0}.
\end{align*}
Recall \eqref{oldle13},
$$
\forall x_0>0,\quad  \int  \left( u_x^2+u^2\right) (0,x)\phi(\sqrt{\sigma_1} (x-x_0)) dx  \leq C e^{- \sqrt{
\sigma_1} x_0}.
$$
Since   $\phi \geq \frac 12$ on $\RR^+$ and since 
  for all $y\in \RR$, $\phi(\sqrt{\sigma_2} y) \leq 2 \phi(\sqrt{\sigma_1} y)$ (by $\sigma_2\geq \sigma_1$), 
we obtain
\begin{align*}
&  \int_{x>t+x_0 } \left( u_x^2 + u^2\right) (t,x) dx 
\leq 2  \int  \left( u_x^2 + u^2\right) (t,x) \phi(\sqrt{\sigma_2} (x-t-x_0)) dx
\\ & \leq 4 \int (u_x^2+u^2)(0,x) \phi(\sqrt{\sigma_2} (x-(1-\sigma_3)t-x_0)) dx + C e^{- \sqrt{\sigma_2} x_0}\\
&\leq 8 \int (u_x^2+u^2)(0,x) \phi(\sqrt{\sigma_1} (x-(1-\sigma_3)t-x_0)) dx + C e^{- \sqrt{\sigma_2} x_0}\\
& \leq  C e^{- \sqrt{\sigma_1} \,(x_0+(1-\sigma_3) t)}+ C e^{- \sqrt{\sigma_2} x_0}.
\end{align*}
Estimate \eqref{infty2vrai} then follows from
$$
\|u\|_{L^\infty(x>t+x_0)}^2 \leq 2 \|u_x\|_{L^2(x>t+x_0)} \|u\|_{L^2(x>t+x_0)}\leq \|u_x\|_{L^2(x>t+x_0)}^2 +  \|u\|_{L^2(x>t+x_0)}^2  .
$$

\medskip

\noindent{\it Step 3.} End of the proof of Lemma \ref{le:droiteN}.
We first claim the following technical facts.
\begin{claim}\label{cl:bm}
\be\label{clbm1}
\frac {\sigma_0}{\gamma_0} \geq \sqrt{c_{j_0+1} c_{j_0}}.
\ee
\be\label{clbm0}
\frac { 4\  \sigma_0}{\frac {\sigma_0}{\gamma_0} -(1-c_{j_0})} <1,
\ee
\be\label{clbm}
   \frac {  \sigma_0}{\gamma_0} + c_{j_0} -\left( \max \left(\frac {4}{5} , \frac { 4  \sigma_0}{\frac {\sigma_0}{\gamma_0} -(1- c_{j_0})} \right)\right)^2   >  5  \sigma_0.
\ee
\end{claim}

Assume Claim \ref{cl:bm}. 
Let 
$$
\sigma_1 = \frac {16}{25} < c_N^- \quad \hbox{by Lemma \ref{sp}},
$$
$$\sigma_2 =     \left( \max \left(\frac 45, \frac { 4\bar \sigma_0}{\frac {\sigma_0}{\gamma_0} -(1-c_{j_0})}\right)\right)^2 <1,\quad 
\sigma_1\leq \sigma_2<\sigma_3<1, \quad \bar \sigma_0 > \sigma_0,
$$
where by \eqref{clbm} and by continuity, we fix $\bar \sigma_0> \sigma_0$ close enough to $\sigma_0$ and $\sigma_3> \sigma_2$
close enough to $\sigma_2$ so that
\be\label{fff}
   \frac { \sigma_0}{\gamma_0} + c_{j_0}  -\sigma_3 > 5 \bar  \sigma_0.
\ee

Applying \eqref{infty2vrai} with $x=x_0(t) -t$ where $x_0(t)= \left( \frac {\sigma_0}{\gamma_0} + c_{j_0}\right) t -K_0$, we obtain using \eqref{fff},
\begin{align*}
|u(t,x_0(t))  | & \leq  
 C   e^{-\frac 25 ( x_0(t) -\sigma_3  t) }+  C e^{- \frac 12 \sqrt{\sigma_2} ( x_0(t)-t) } \\
 &  \leq C e^{-\frac 25 \left[(   \frac { \sigma_0}{\gamma_0} + c_{j_0}  -\sigma_3 )t-K_0\right]} + Ce^{- \frac 12 \sqrt{\sigma_2} \left[( \frac {\sigma_0}{\gamma_0} + c_{j_0} -1)t -K_0\right] } \\
 &  \leq C(K_0)  e^{-2 \bar \sigma_0 t} \leq e^{-2 \sigma_0 t},
\end{align*}
for all $t>t_0(K_0)$, provided $t_0(K_0)$ is large enough.

\begin{proof}[Proof of Claim \ref{cl:bm}]
First, by explicit computations, we see that  
\be\label{bb}
0\leq c\leq 1 \quad \Rightarrow \quad 
\sqrt{1-\frac 34 c}  \geq 1- \frac {\sqrt{c}}2,
\ee
so that $\gamma(c) : = \sqrt{1-\frac 34 c}   - \frac {\sqrt{c}}2  \geq 1-\sqrt{c}$.
Moreover,
$$
\gamma(c) = \frac {1-c}{\sqrt{1-\frac 34 c} + \frac {\sqrt{c}}{2}}\leq 1-c,
$$
thus
\be\label{gc}
\forall c\in [0,1],\quad 
1-\sqrt{c} \leq  \sqrt{1-\frac 34 c}   - \frac {\sqrt{c}}2  \leq 1-c.
\ee

Since $\gamma_0 = \sqrt{c_{j_0}} \gamma(c_{j_0+1}/c_{j_0})$, we obtain
$$
\frac {\sigma_0}{\gamma_0} \geq  \frac {\sqrt{c_{j_0+1}} (c_{j_0}-c_{j_0+1})}{\sqrt{c_{j_0}}  (1-c_{j_0+1}/c_{j_0})}
= \sqrt{c_{j_0+1} c_{j_0}}.
$$

Next, we prove \eqref{clbm0}. Observe that 
\be\label{gsz}
\sigma_0 = c_{j_0} \gamma_0 - \gamma_0^3,
\ee
and, since $\frac 34 < c_j \leq 1$ for all $j$,
\be\label{bsz}
0<\gamma_0 < \sqrt{c_{j_0}} \left( \sqrt{1-\frac 9{16}} - \frac {\sqrt{3}} 4\right) = \frac {\sqrt{c_{j_0}}}{4} (\sqrt{7}-\sqrt{3}).
\ee
Since $1-c_{j_0+1} = 1-c_{j_0} + c_{j_0} - c_{j_0+1} \leq \frac 14$, we  obtain
\begin{align*}
4 \sigma_0  & = 4 \sqrt{c_{j_0+1}} (c_{j_0} - c_{j_0+1}) \leq 4 \sqrt{c_{j_0+1}} \left(\frac 14 - (1-c_{j_0} )\right) \\
& \leq \sqrt{c_{j_0+1}} - 4 \sqrt{c_{j_0+1}} (1-c_{j_0}) \leq c_{j_0}Ê\sqrt{c_{j_0+1}}  - 3 \sqrt{c_{j_0+1}} (1-c_{j_0})\\
& \leq c_{j_0}Ê\sqrt{c_{j_0+1}}  - 3 \frac {\sqrt{3}}2   (1-c_{j_0}) \leq  \sqrt{c_{j_0+1}} \sqrt{c_{j_0}}   -   (1-c_{j_0}) \leq \frac {\sigma_0}{\gamma_0} - (1-c_{j_0}).
\end{align*}

Finally, we prove \eqref{clbm}. We begin with the case where $\frac {\sigma_0}{\frac {\sigma_0}{\gamma_0} -(1- c_{j_0})}< \frac 15$.
Then, necessarily $\gamma_0 < \frac 15$. It is then clear that, using \eqref{gsz},
$$  \frac {\sigma_0}{\gamma_0} +c_{j_0} - \frac {16}{25} - 5 \sigma_0
  =   {c_{j_0}} - \frac {16}{25} + (c_{j_0} - \gamma_0^2) (1-5 \gamma_0)\geq 0.
$$

Second, we assume $\frac {\sigma_0}{\frac {\sigma_0}{\gamma_0} -(1- c_{j_0})}\geq  \frac 15$. 
We distinguish two cases depending on the value of $j_0$.

If  $j_0=1$ then $c_{j_0}=1$ and we are reduced to prove 
\be\label{clbmee}
\frac  {  \sigma_0}{\gamma_0}  + 1  -16 \gamma_0^2   >  5  \sigma_0  
\ee
where $\frac {\sigma_0}{\gamma_0}=1-\gamma_0^2$, so that it is sufficient to have
$$
0< \gamma_0 \leq \frac {\sqrt{7}-\sqrt{3}}{4} \quad \Rightarrow \quad 
2- 5 \gamma_0 - 17 \gamma_0^2 +5 \gamma_0^3 >0 ,$$
which is easily checked by explicit computations.

We now consider the case where $j_0\geq 2$ and thus $N\geq 3$.
Let
$$
\tilde \sigma_0 = \frac {\sigma_0}{c_{j_0}^{\frac 32}},\quad
\tilde \gamma_0 = \frac {\gamma_0}{c_{j_0}^{\frac 12}},\quad
\frac {\tilde \sigma_0}{\tilde \gamma_0}=1-\tilde \gamma_0^2
$$
Then, \eqref{clbm} is equivalent to
\be\label{clbm3}
\frac { \tilde \sigma_0}{\tilde \gamma_0} + 1 -\left( \frac { 4  \tilde \sigma_0}{\frac {\tilde \sigma_0}{\tilde \gamma_0} -(\frac 1{c_{j_0}} -1 )}\right)^2   >  5 \sqrt{c_{j_0}} \tilde \sigma_0.
\ee
It is clear that the following inequality
\be\label{fol}
\frac { \tilde \sigma_0}{\tilde \gamma_0} + 1 -\left( \frac { 4  \tilde \sigma_0}{\frac {\tilde \sigma_0}{\tilde \gamma_0} -\frac 13 }\right)^2   >  5   \tilde \sigma_0.
\ee
would imply \eqref{clbm3}.
By explicit computations, one can see that \eqref{fol} is not satisfied for all $\tilde \gamma_0 \in (0,\frac {\sqrt{7}-\sqrt{3}}{4}]$. At this point, we need to use the definition of $j_0$ and the fact that 
$j_0\geq 2$ to lower the value of $\tilde \gamma_0$ for which we have to check \eqref{fol}.

Indeed, by the definition of $j_0$, we have
$$
\frac{\sqrt{3}}2 (c_{j_0}-c_{j_0+1})\leq 
\sqrt{c_{j_0+1}} (c_{j_0}-c_{j_0+1}) \leq \sqrt{c_2} (1-c_2)\leq (1-c_2).
$$
Since $1-c_2 + c_{j_0} - c_{j_0+1} \leq \frac 14$, we obtain
$
c_{j_0}-c_{j_0+1} \leq \frac 1{2 (2+\sqrt{3})}$ and so by simple computations,
$$
\frac {c_{j_0+1}}{c_{j_0}} \geq \frac 1{1+\frac {2}{3(2+\sqrt{3})}} \quad \hbox{which implies} \quad
\tilde \gamma_0 \leq \frac 3{20}.
$$
Moreover, we easily check  that \eqref{fol} is indeed satisfied for all $\tilde \gamma_0\in [0,\frac 3{20}]$.
\end{proof}

\section{Construction and lower bound of the approximate solution}\label{sec:new4}
This section contains the main ingredient of the proof, i.e. the construction of an approximate solution for $t\gg 1$,   and the description of its asymptotics for $x\gg 1$.
The approximate solution is built by explicit resolution of the  main contribution of the   interactions of the solitons for $t\to +\infty$, i.e. where solitons are decoupled. Note that such refined computations were not needed for existence result.

\medskip

Recall the following notation.
Let  $j_0\in \{1,\ldots,N-1\}$ be such that
\be\label{defjz4}
\sqrt{c_{j_0+1}}(c_{j_0}-c_{j_0+1}) = \min_j \sqrt{c_{j+1}}(c_j-c_{j+1}):=\sigma_0
\ee
and
\be\label{defgamma4}
\gamma_0=  \sqrt{ {c_{j_0}} - \frac 34   {c_{j_0+1}}} - \frac 12 \sqrt{  {c_{j_0+1}} }, \quad  
x_0(t)=\left(\frac {\sigma_0}{\gamma_0}   + c_{j_0}\right) t - K_0,
\ee
where $K_0>0$.

\begin{proposition}[Approximate solution]\label{pr:2}
Assume \eqref{speeds}. 
There exists a function $V(t,x)$ such that
\begin{enumerate}[\rm 1.]
\item  There exists $C>0$ such that, for all  $t\geq 0$,
\be\label{eq:bar}
\bigg\|V(t) - \sum_{j=1}^N Q_{c_j}(x-c_jt-\Delta_j)\bigg\|_{H^3}\leq Ce^{- \sigma_0 t}.
\ee
\item For all  $t\geq 0$,
\be\label{eq:eqV}
\left\| \partial_t V + \partial_x(\partial_x^2 V + V^4)   \right\|_{H^3} \leq C e^{-2\sigma_0 t}.
\ee
\item There exist $\kappa>0$ and $t_1=t_1(K_0)>0$, such that, for all  $t\geq t_1$, 
\be\label{eq:lowV}
 |V(t,x_0(t) )  |\geq \kappa  e^{\gamma_0 K_0} e^{-2 \sigma_0 t}.
\ee
\end{enumerate}
\end{proposition}

This section is devoted to the proof of Proposition \ref{pr:2}.

\subsection{ODE analysis}
For $c>0$, let  
\be
\label{deflplus}
L_c f=-f''+c f-4 Q_c^3 f,\quad L=L_1,
\ee 
$$
\Lambda Q_c = \frac 1 c \left(\frac 13 Q_c + \frac 12 x Q_c'\right) = \frac d{d{c'}} {Q_{c'}}_{|{c'=c}},\quad
\Lambda Q= \Lambda Q_1, \quad  L \Lambda Q = - Q.
$$
Denote the $L^2$ scalar product by $(f,g)=\int f(x)g(x)dx.$

\medskip

Following the outline of the proof in the Introduction, we define the solution of equation \eqref{ode} related to the main perturbative terms in the construction of the approximate solution.

\begin{lemma}\label{pr:hypo2} Let $\frac 34 \leq c <1$. 
There exists a unique solution $A_c\in H^{10}$ of
\be\label{eq:AAA}
(L A_c)'  + \sqrt{c}(1-c) A_c   =G_c' \quad \hbox{where }G_c(x)=e^{-\sqrt{c} x} Q^3(x).
\ee
Moreover, $A_c$ satisfies the following, for all $k\geq 0$,
\begin{itemize}
\item[{\rm (i)}] {\rm Decay estimates.}
\begin{equation}\label{Aright}
|A_c^{(k)}(x)|\lesssim  e^{\gamma_c^0 x} \quad \hbox{for $x<0$,}\qquad 
|A_c^{(k)}(x)|\lesssim e^{-\gamma_c^{\rm I} x} \quad \hbox{for $x>0$,}
\end{equation}
where  
$$
\gamma_c^0 = \sqrt{ 1 - \frac 34 c}+ \frac 12\sqrt{c} ,
\quad
\gamma_c^{\rm I} =   \sqrt{ 1 - \frac 34 c} - \frac 12\sqrt{c},
\quad
\gamma_c^{\rm II} = \sqrt{c},
$$
$$
0<\gamma_c^{\rm I}< \gamma_c^{\rm II} = \sqrt{c} < \gamma_c^0 = \gamma_c^{\rm I} + \gamma_c^{\rm II}.
$$
\item[{\rm (ii)}] {\rm Asymptotics at $+\infty$.} There exist $a_c^{\rm I,II}\in \RR$ such that, for all $k\geq 0$,
$$\left|(d/dx)^k\left(A_c(x) - a_c^{\rm I} e^{-\gamma_c^{\rm I} x}-  a_c^{\rm II} e^{- \gamma_c^{\rm II} x}\right)\right|\lesssim  e^{-\frac 32 |x|} \quad \hbox{for $x>0$.}$$
\item[{\rm (iii)}] {\rm Generic exponential decay on the right-hand side.}
\begin{equation}\label{eq:generic}
 a_c^{\rm I} =\lim_{x \to +\infty}A_c(x) e^{\gamma_c^{\rm I} x}\neq 0.
\end{equation}
\item[{\rm (iv)}] {\rm Continuity of $a_c^{\rm I}$.} 
The function $c\in [\frac 34,1) \mapsto a_c^{\rm I}$ is continuous.
In particular, $a_c^{\rm I}$ has a constant sign on $[\frac 34,1)$.
\end{itemize}
\end{lemma}

\begin{remark}\label{ccrit}
Note   that $\gamma_c^0$, $-\gamma_c^{\rm I}$ and $-\gamma_c^{\rm II}$ are the  three roots of
$
\gamma^3 - \gamma -\sqrt{c} (1-c)=(\gamma+\sqrt{c})(\gamma^2 -\sqrt{c} \gamma -(1-c))=0.
$
Since $\frac 34 \leq  c<1$, we have 
\be\label{gammaZ} \sqrt{c} < \gamma_c^0 < 2\sqrt{c}, \quad 0<\gamma_c^{\rm I}<\sqrt{c}.\ee
It is a key point in our proof to be able to prove that  $A_c$ has generic decay, i.e.     the exponential decay $e^{-\gamma_c^{\rm I} x}$ for 
$x>0$.

For $0<c\leq \frac 13$,  the situation is different since $\gamma_c^{\rm I}\geq \sqrt{c}$. In particular,
$e^{-\sqrt{c} x}$ is the generic decay of $A_c$ in this case. 
The strategy of this paper cannot be applied directly.
\end{remark}

\begin{remark}\label{ttg}
By Claim \ref{cl:bm},
\be\label{futur}
1-\sqrt{c} \leq \gamma_c^{\rm I} \leq 1-c.
\ee
\end{remark}

\begin{proof}
{\bf Proof of (i)-(ii).}
First, note that for  $0<\theta<\frac 2{3\sqrt{3}}$, equation $\gamma^3-\gamma-\theta=0$ has three distinct real roots: $\gamma_\theta^0$, $-\gamma_\theta^{\rm I}$ and  $-\gamma_\theta^{\rm II}$, where
$$
0\leq \gamma_\theta^{\rm I}<\gamma_\theta^{\rm II} < 1 <\gamma_\theta^{0}=\gamma_\theta^{\rm I}+\gamma_\theta^{\rm II} \leq \frac 2{\sqrt{3}}.
$$
From  the spectral analysis of Pego and Weinstein \cite{PW} and standard ODE arguments, one has the following general result.
\begin{claim} \label{PG}\label{decayexpo} {\rm (a)  Existence.}
For all $\theta>0$, for all $F\in L^2$,  there exists a  unique solution $  A\in H^3$ of 
\be\label{eq:nh}
(L    A)'  + \theta   A   = F.
\ee
{\rm (b)  Decay.}
Let $F$ be a $C^\infty$ function such that
\be\label{decayH}
\forall k\geq 0,\forall x\in \RR, \quad |F^{(k)}(x)|\leq C_k e^{-\frac 32 |x|}.
\ee  Assume $0<\theta<\frac 2{3\sqrt{3}}$ and 
let $A$ be the    solution of \eqref{eq:nh}. Then, $A$ is $C^\infty$ and satisfies, for all $k\geq 0$,

\noindent{$\bullet$ \rm Decay estimates.}
\begin{equation}\label{Arightdd}
|A^{(k)}(x)|\lesssim  e^{\gamma_\theta^0 x} \quad \hbox{for $x<0$,}\qquad 
|A^{(k)}(x)|\lesssim e^{-\gamma_\theta^{\rm I} x} \quad \hbox{for $x>0$.}
\end{equation}
$\bullet$ {\rm Asymptotics at $+\infty$.} There exists $a_\theta^{\rm I,II}\in \RR$ such that
$$\left|(d/dx)^k\left(A(x) - a_\theta^{\rm I} e^{-\gamma_\theta^{\rm I} x}-  a_\theta^{\rm II} e^{- \gamma_\theta^{\rm II} x}\right)\right|\lesssim  e^{-\frac 32 |x|} \quad \hbox{for $x>0$.}$$
\end{claim}
\begin{proof}[Proof of Claim \ref{PG}] Proof of (a).
In \cite{PW}, it is proved  that no real nonzero eigenvalue of the operator
$(L    A)' $ exists. We reproduce the proof here for the sake of completeness.
Recall the following basic facts on $L$ (see \cite{We})
\begin{itemize}
\item  ${\rm Ker} \, L = \{\lambda Q', \lambda \in \mathbb{R}\}$;
\item  There exists $\mu>0$ such that for any $f\in H^1$,
\begin{equation}\label{eq:we}
	(f,Q)=(f,Q')=0 \quad \Rightarrow \quad  (Lf,f)=\int (f')^2+f^2 - 5Q^4 f^2 \geq \mu \|f\|_{H^1}^2.
\end{equation}
\end{itemize}
Let $A$ be a solution of $(LA)'+\theta A=0$. Let $\bar A = A + a Q'$, where $a$ is such that $\int \bar A Q'=0$.
Then, since $LQ'=0$,
\be\label{eq:A}(L\bar A)' + \theta \bar A= a \theta Q'.\ee
Taking the scalar product of \eqref{eq:A} by $Q$, we obtain $(\bar A,Q)=0$. Moreover, taking the scalar product of \eqref{eq:A}
by $L\bar A$, we get $(L\bar A,\bar A)=0$. By \eqref{eq:we}, it follows that $\bar A=0$, and so $A=- aQ'$. But then
$\theta A=0$, so that $\theta =0$ or $A\equiv 0$. 
\smallskip

Once we know that no real nonzero eigenvalue exists, the invertibility of $(LA)'+\theta A$ follows from usual arguments (Fredholm alternative).
 \medskip
 
 Proof of (b). These properties follow from standard ODE arguments.
\end{proof}

Apply Claim \ref{PG} to $F=G_c'$ and $\theta=\sqrt{c}(1-c)$. There exists a unique $  A_c\in H^3$ such that 
$(L  A_c)' +\sqrt{c}(1-c)    A_c=G_c'$.
Note that $\sqrt{c}(1-c) < \frac 2{3\sqrt{3}}$, and thus (i) and (ii) are direct consequences of Claim \ref{decayexpo} and standard regularity arguments.

  \medskip

 {\bf Proof of (iii).}  This point is more delicate. Let $\frac 34<c<1$. Assume for the sake of the contradiction that 
$a_c^{\rm I}=0$, so that by (ii)
\be\label{decay}
\forall x>0, \quad |A_c^{(k)}(x)|\lesssim e^{- \sqrt{c} |x|} \lesssim e^{-\sqrt{\frac 34} |x|}.
\ee

\emph{Step 1.} Reduction to a dual equation. 
Let $  a$, $  b$  be such that $  A = A_c -   a \Lambda Q -   b Q'$ satisfies 
\begin{equation}\label{orthoA}
\int   A \, Q^{\frac 52} = 0, \quad  \int   A Q'=0.
\end{equation}
Note that $L Q^{\frac 52} = - \frac {21}{4} Q^{\frac 52}$ so that $  a$ exists since $( Q^{\frac 52},\Lambda Q)=-\frac 4{21}(LQ^{\frac 52}, \Lambda Q )=\frac 4{21} (Q^{\frac 52} ,Q)  \neq 0$.

Then by $L \Lambda Q = -Q$, 
$ A$ satisfies 
$$(L   A)' +\sqrt{c}(1-c)    A 
+ \sqrt{c}(1-c)   a \Lambda Q
- (a -\sqrt{c} (1-c) b) Q'=G_c'.$$
Set $a_0 =   \sqrt{c}(1-c) a$, $B= - L  A$. Then,
\be\label{eq:B}
L(B') +\sqrt{c}(1-c) B + a_0 Q = - L (G_c'),\quad \int B Q^{\frac 52}= 
\int B Q'=0.
\ee
We decompose 
$$G_c =  e^{-\sqrt{c} x} Q^3 =  G_0 +\alpha Q' + \beta Q,$$
 where 
$$
\alpha = \frac {\int G_c Q'}{\int (Q')^2}\neq 0 ,\quad \beta = \frac {\int G_c Q^{\frac 52}}{\int Q^{\frac 72}},\quad
\int G_0 Q^{\frac 52}=\int G_0Q'=0.
$$
Now, we set $  B_0=B+G_0$. We have thus proved 
\begin{claim}\label{claimB0}
Assuming \eqref{decay}, there exists a   smooth function $B_0\in H^3$ such that 
\be\label{eq:B0}
L(B_0') +\sqrt{c}(1-c) B_0 + a_0 Q = - \alpha L (Q'')+\sqrt{c}(1-c) G_0,\ee
\begin{equation}\label{orthoB0}
\int B_0 Q^{\frac 52} = \int B_0 Q'=0,
\end{equation}
\be\label{decayB0}
\forall k\geq 0, \ \forall x<0, \quad  |B_0^{(k)}(x)|\lesssim e^{-\gamma_c^0 x},\quad 
 \forall x>0,\ |B_0^{(k)}(x)|\lesssim e^{-\sqrt{\frac  34} x}.
\ee
\end{claim}

\emph{Step 2.} We prove by Virial type arguments that such $B_0$ does not exist.
\smallskip

\emph{-- Computation of $a_0$ and $\alpha$ from \eqref{eq:B0}--\eqref{orthoB0}.}
On the one hand, we multiply \eqref{eq:B0} by $\Lambda Q$ and   use $L \Lambda Q =-Q$, $(B_0,Q')=0$,
$(Q,\Lambda Q)= \frac 1{12} \int Q^2$, so that
$$
\sqrt{c} (1-c) (B_0,\Lambda Q) + \frac {a_0} {12}\int Q^2 = - \alpha \int (Q')^2 + \sqrt{c}(1-c) \int G_0\Lambda Q,
$$
and thus (using $\int (Q')^2 = \frac 37 \int Q^2$)
\begin{equation}\label{surabis}
a_0+ \frac {36}{7} \alpha  - \frac {12}{\int Q^2} \sqrt{c}(1-c)  (G_0-B_0,\Lambda Q) =0.
\end{equation}

On the other hand,   consider $H_0\in L^\infty$ such that $H_0'\in H^2$, $LH_0 = 1$, $(H_0,Q')=0$.
An explicit expression of $H_0$ is available in \cite{MMcol1}, Claim 3.1: $$H_0=  1+ \frac 13 \left( Q' \int_0^x Q^2- 2 Q^3\right).$$
Multiplying \eqref{eq:B0} by $H_0$ and using $\int B_0'=\int Q''=0$, we find
$$
\sqrt{c}(1-c) \int B_0H_0+ a_0 \int QH_0=\sqrt{c}(1-c) \int G_0 H_0,
$$
so that (note that $\int QH_0 = -\int (L\Lambda Q) H_0 = -\int \Lambda Q =  \frac 16 \int Q$).
\begin{equation}\label{suratri}
a_0= \frac 6{\int Q} \sqrt{c}(1-c)(G_0-B_0,H_0).
\end{equation}

From \eqref{surabis} and \eqref{suratri}, we deduce
\begin{equation}\label{suralpha}
  \alpha =-\frac 73\sqrt{c}(1-c) (G_0-B_0,J_0)\quad \hbox{where} \quad 
J_0=\frac {H_0}{2\int Q}   - \frac {\Lambda Q}{\int Q^2}.
\end{equation}

\emph{-- Estimate on $B_0$ by Virial type identity.}
We adapt a strategy developed in \cite{yvanSIAM}, \cite{MMas1}.
Multiply   equation \eqref{eq:B0} by $B_0 \frac {Q'}{Q^2}$ and integrate.
Note that all the integrals below are well-defined because of the decay properties \eqref{decayB0}. Then, using
$$
L(Q'')= (LQ')'+4(Q^3)' Q'=12 (Q')^2 Q^2,
$$
we have
\begin{align}
&\int L(B_0') B_0 \frac {Q'}{Q^2}+ \sqrt{c}(1-c) \int B_0^2 \frac {Q'}{Q^2} + a_0 \int B_0 \frac {Q'} Q \nonumber \\
&=-12  \alpha \int B_0 (Q')^3 +\sqrt{c}(1-c) \int B_0 \frac {Q'}{Q^2} G_0 . \label{prems}
\end{align}
The key argument to obtain a contradiction is a coercivity property of the quadratic form $\int L(B_0') B_0 \frac {Q'}{Q^2}$ under the orthogonality conditions \eqref{orthoB0}.
\begin{claim}\label{le:B}
$$\int L(B_0') B_0 \frac {Q'}{Q^2}
\geq \frac 38 \int \frac {B_0^2}{Q} + 7  \int B_0^2 Q^2.$$
\end{claim}

See proof of Claim \ref{le:B} in Appendix \ref{ap:A}.

\medskip
Note that $(Q')^2 = Q^2 - \frac 25 Q^5 \leq Q^2$, so that $\frac {|Q'|}{Q^2}\leq\frac 1 Q$, and thus
  for $\frac 34 \leq   c <1$, we have
\begin{equation}\label{cnv}
\int L(B_0') B_0 \frac {Q'}{Q^2}+ \sqrt{c}(1-c) \int B_0^2 \frac {Q'}{Q^2}  
\geq \int  B_0^2 F_0 ,
\end{equation}
where
$$ F_0= \frac {3-   {\sqrt{3}}}{8} \frac 1 Q + 7  Q^2 . $$
Define 
$$N^*(B_0): =\left(\int B_0^2 F_0\right)^{\frac 12}.
$$
Define the operator $P$, the projection onto the orthogonal  of $\hbox{span}(Q^{\frac 52};Q')$ for the scalar product
$\int (fg/F_0)$.
In particular, for a given function $f$ such that $\int |f|^2 Q <+\infty$,  we have
$$
\left|\int B_0 f\right| \leq N^*(B_0) N(f) \quad \hbox{where} \quad
N(f) = \left( \int \frac {(Pf)^2}{F_0}\right)^{\frac 12}.
$$

We claim
\begin{claim}\label{cl:mis}
\begin{align}
& N^*(B_0)  \leq \sqrt{c}(1-c) \frac {k_1(c)}{k_2(c)}, \quad \hbox{where} \label{mis}\\
&  k_1 (c)=  N\left( - \frac 6{\int Q}(G_0,H_0) \frac {Q'}{Q} +28(G_0,J_0)(Q')^3+\frac {Q'}{Q^2 }G_0\right),\nonumber\\
& k_2(c) = 1- \sqrt{c}(1-c) \left( \frac {3}{\int Q}  N(H_0) N\left(\frac {Q'}{Q}\right) + 14 N(J_0) N\left((Q')^3\right) \right).\nonumber
\end{align}
\end{claim}

\begin{proof}[Proof of Claim \ref{cl:mis}]
As a consequence of \eqref{prems},  \eqref{cnv}, \eqref{suratri}, \eqref{suralpha}, we get
\begin{align}
N^*(B_0)^2 &
\leq   - \frac 6{\int Q}\sqrt{c}(1-c)  (G_0-B_0,H_0) \left( B_0, \frac {Q'}{Q}\right) \nonumber\\
& 
+{28} \sqrt{c}(1-c) (G_0-B_0,J_0) \left(B_0 ,(Q')^3\right)+\sqrt{c}(1-c)   \int B_0  \frac {Q'}{Q^2}G_0.\label{firstXzerobis}
\end{align}
Note that by parity properties
$$
\left|(B_0,H_0) \left( B_0, \frac {Q'}{Q}\right) \right| \leq 
\frac 12 N^*(B_0)^2 N(H_0) N\left(\frac {Q'}{Q}\right),
$$
$$
\left|(B_0,J_0) \left( B_0, (Q')^3\right) \right| \leq 
\frac 12 N^*(B_0)^2 N(J_0) N\left((Q')^3\right).
$$
Thus,
\begin{align}
& N^*(B_0) \left[ 1- \sqrt{c}(1-c) \left( \frac {3}{\int Q}  N(H_0) N\left(\frac {Q'}{Q}\right) + 14 N(J_0) N\left((Q')^3\right) \right)\right] \nonumber\\
& \leq \sqrt{c}(1-c) N\left( - \frac 6{\int Q}(G_0,H_0) \frac {Q'}{Q} +28(G_0,J_0)(Q')^3+\frac {Q'}{Q^2}G_0
  \right),
\end{align}
and \eqref{mis} is proved.
\end{proof}
 
\emph{-- Conclusion.}
From   \eqref{suralpha} and Claim \ref{cl:mis},
$$
\left| \frac 37 \alpha +\sqrt{c}(1-c) (G_0,J_0) \right| \leq  \sqrt{c}(1-c)   N(J_0) N^*(B_0)
\leq c (1-c)^2  N(J_0) \frac {k_1(c)}{k_2(c)}.
$$
Thus,
\begin{align}
& k(c):=   \frac {c(1-c)^2  N(J_0)  }{ \left|\frac {3}7  \alpha +   \sqrt{c}(1-c) (G_0,J_0)\right|}\frac {k_1(c) } {k_2(c)}  \geq 1.
\label{def:kc}
\end{align}
Observe that $k(c)$ is defined   through explicit functions of $c$ (and does not depend on the function $B_0$). Therefore, one can compute $k(c)$ directly by various integrations of explicit functions.
We check numerically that for all $c\in [\frac 34,1]$, $0\leq k(c)\leq 0.5 <1$. In particular,  a contradiction arises from \eqref{def:kc} for all $c\in [\frac 34,1]$ as desired.

\medskip

\noindent{\bf Proof of (iv).}
Let $c,\tilde c\in (\frac 34,1)$. Let $A_c$, $A_{\tilde c}$ be the corresponding solutions of \eqref{eq:AAA}.

\smallskip

First, we claim, for $C=C(c)$, $\tilde c$ close to $c$,   
\be\label{eq:smallA}
\| A_c-A_{\tilde c}\|_{H^3}\leq  C |c-\tilde c|.
\ee
 Indeed, let $\epsilon = \sqrt{c}(1-c)-\sqrt{\tilde c} (1-\tilde c)$, $G_\epsilon(x) = \left( e^{-\sqrt{c} x} - e^{-\sqrt{\tilde c} x}\right) Q^3(x)$. 
Then,
\be\label{eq:AAt}
(-L(A_c-A_{\tilde c}))' -\sqrt{\tilde c}(1-\tilde c) (A_c-A_{\tilde c}) = \epsilon   A_c   +G_\epsilon.
\ee
Let $A_c-A_{\tilde c} = D + a Q'$ so that $\int D Q'=0$ and
$$
(-LD)' - \sqrt{\tilde c}(1-\tilde c)  D = \epsilon   A_c   +G_\epsilon 
+ a \sqrt{\tilde c}(1-\tilde c) Q'.
$$
Multiplying the equation by $Q$ and integrating, we find
$
|(D,Q)|\leq C |c-\tilde c|.
$
Multiplying the equation by $L D$ and integrating, we find
$
|(L D,D) |\leq C |c-\tilde c| \|D\|_{L^2},
$
so that by \eqref{eq:we},
$\|D\|_{H^1} \leq C |c-\tilde c|,$
where $C$ depends on $c$. Multiplying the equation of $D$ by $Q'$ and integrating, we find
$|a|\leq C |c-\tilde c|$. Next, by the equation of $D$, $\|D\|_{H^3}\leq C|c-\tilde c|$, which implies
 $\|A_c-A_{\tilde c}\|_{H^3} \leq C |c-\tilde c|$.

\smallskip

Second, we set $a_c^{\rm I} =\lim_{x \to +\infty}A_c(x) e^{\gamma_c^{\rm I} x}\neq 0$,
 $a_{\tilde c}^{\rm I} =\lim_{x \to +\infty} A_{\tilde c}(x) e^{\gamma_{\tilde c}^{\rm I} x}$ and
we prove that
\be\label{eq:cg}
\lim_{\tilde c\to c} a_{\tilde c}^{\rm I}= a_c^{\rm I}.
\ee
Let $0<\delta<\frac 1{10}$ arbitrary.  Fix $x_0>0$ large enough so that
\begin{align}
& \sum_{j=0}^2  
 \left|A_c^{(j)}(x_0) - (-\gamma_c^{\rm I})^j  a_c^{\rm I} e^{-\gamma_c^{\rm I} x_0} \right|
\leq \delta e^{-\gamma_c^{\rm I} x_0},
\label{sA} \\ 
&\forall x>x_0,\quad 
 Q^2(x) \leq \delta e^{-\frac 32 x}. \label{sq}
\end{align}
From \eqref{eq:smallA} and continuity of $\gamma_c^{\rm I}$ in $c$, we take $|\tilde c -c|$ small enough so that
\be\label{closeA}
\sum_{j=0}^2\left|A_c^{(j)} (x_0)-A_{\tilde c}^{(j)}(x_0)\right|
+|e^{-\gamma_c^{\rm I} x_0}- e^{-\tilde \gamma_c^{\rm I} x_0}|+|\gamma_c^{\rm I} - \tilde \gamma_c^{\rm I}|
 \leq \delta e^{-\gamma_c^{\rm I} x_0}.
\ee
Then by \eqref{sA},
\be\label{tAi}
 \sum_{j=0}^2  
\left|A_{\tilde c}^{(j)}(x_0) - (-\tilde \gamma_c^{\rm I})^j   a_c^{\rm I} e^{-\tilde \gamma_c^{\rm I} x_0} \right|
\leq C \delta e^{-\tilde \gamma_c^{\rm I} x_0}.
\ee
From the equation of $A_{\tilde c}$, \eqref{eq:smallA} and \eqref{sq} (which implies
for $x\geq x_0$, $|(Q^3 A_c)'|+|G_{\tilde c}|\leq C \delta e^{-\frac 32 x}$), we have 
\be\label{tAe} 
\forall x>x_0,\quad 
 \left| A_{\tilde c}'''- A_{\tilde c}' - \sqrt{\tilde c} (1-\tilde c) A_{\tilde c} \right| 
\leq C \delta e^{-\frac 32 x}.
\ee
Now, it follows from \eqref{tAi}, \eqref{tAe} and standard ODE arguments (Duhamel formula) that
\be\label{tAc}
\forall x>x_0,\quad
\left| A_{\tilde c}(x) -   a_c^{\rm I} e^{-\tilde \gamma_c^{\rm I} x} \right| \leq
C \delta e^{-\tilde \gamma_c^{\rm I} x}
\ee
and thus
$
|a_c^{\rm I} - a_{\tilde c}^{\rm I}|\leq C \delta.
$
\end{proof}

The construction of the approximate solution requires the introduction of solutions of other ODEs but no refined property of these solutions is needed.
We state two lemmas similar to Lemma \ref{pr:hypo2} (i)-(ii), whose proofs are direct consequences of Claim~\ref{PG}.

\begin{lemma}\label{le:prbis}
Let $1<c< \frac 43$.
 There exists a unique   solution   $A_c\in H^3$ of 
\be\label{eq:lGbis}
(L A_c)'  + \sqrt{c} (c-1) A_c   =G_c'  \quad \hbox{where }G_c(x)=e^{\sqrt{c} x} Q^3(x),
\ee
such that
\begin{equation}\label{Arightbis}
|A_c^{(k)}(x)|\lesssim  e^{\gamma_c^0 x} \quad \hbox{for $x<0$,}\qquad 
|A_c^{(k)}(x)|\lesssim e^{ -   \gamma_c^{\rm I} x} \quad \hbox{for $x>0$,}
\ee
where
$$
\gamma_c^0= \sqrt{c},\quad
  \gamma_c^{\rm I} =  \frac {   \sqrt{c}}  2 - \sqrt{1 - \frac 34 c}  , \quad 
  \gamma_c^{\rm II} = \frac { \sqrt{c}}   2 + \sqrt{1 - \frac 34 c} ,
$$
$$
0<\gamma_c^{\rm I}< \gamma_c^{\rm II}  < \gamma_c^0 = \gamma_c^{\rm I} + \gamma_c^{\rm II} = \sqrt{c}.
$$
Moreover, there exist $a_c^{\rm I, II}$ such that 
\begin{equation}\label{Aleftqua}
\bigg|(d/dx)^k\left(A_c(x) - a_c^{\rm I} e^{-\gamma_c^{\rm I} x}- a_c^{\rm II} e^{-\gamma_c^{\rm II} x}\right)\bigg|\lesssim e^{-\frac 32 |x|} \quad \hbox{for $x>0$.}
\end{equation}
\end{lemma}
It is easily checked that $\gamma_c^0$, $-\gamma_c^{\rm I, II}$ are the roots of $-\gamma^3+\gamma+\sqrt{c}(c-1)$.

\begin{lemma}\label{le:prtri}
Let $\frac 34\leq c\leq \frac 43$, $c\neq 1$ and $\frac 34\leq c'<1$.
 There exist  unique   solutions   $A_{c,c'}^{\rm I}$, $A_{c,c'}^{\rm II}\in H^3$ of 
\be\label{eq:lGtri}
(L A_{c,c'}^{\rm I,II})'  + \left( c'^{\frac 32} \sqrt{c} |1-c| + \sqrt{c'} \gamma_c^{\rm I,II}(1-c')  \right)  A_{c,c'}^{\rm I,II}   = (G_{c,c'}^{\rm I,II})',\ee
where
\be G_{c,c'}^{\rm I,II}(x)=e^{- \sqrt{c'}\gamma_c^{\rm I,II} x} Q^3(x).
\ee
Moreover,
\begin{equation}\label{Arighttri}
|(A_{c,c'}^{\rm I,II})^{(k)}(x)|\lesssim  e^{(\gamma_{c,c'}+ \gamma_c^{\rm I,II} \sqrt{c'})x} \quad \hbox{for $x<0$,}\qquad 
|(A_{c,c'}^{\rm I,II})^{(k)}(x)|\lesssim e^{ -  \gamma_{c,c'}  x} \quad \hbox{for $x>0$,}
\ee
where
$$
\gamma_{c,c'}  = \sqrt{1-\frac 34 (\gamma_c^{\rm I,II})^2 c'} - \frac {\gamma_c^{\rm I,II} \sqrt{c'}}2 .
$$
Moreover,
\begin{equation}\label{rsto}
\gamma_{c,c'} \geq 1 - \gamma_c^{\rm I,II} \sqrt{c'}.
\end{equation}
\end{lemma}
Note that
$$
  c'^{\frac 32} \sqrt{c} |1-c| + \sqrt{c'} \gamma_c^{\rm I,II}(1-c')
  = \sqrt{c'} \gamma_c^{\rm I,II} \left( 1 - c' (\gamma_c^{\rm I,II})^2 \right).
$$
One then easily checks that the three roots of $- \gamma^3 +\gamma +   c'^{\frac 32} \sqrt{c} |1-c| + \sqrt{c'} \gamma_c^{\rm I,II}(1-c')  $ are
$$
\gamma_{c,c'}+\gamma_c^{\rm I,II} \sqrt{c'}, \ - \gamma_{c,c'}, - \sqrt{c'}\gamma_c^{\rm I,II}.
$$
Equivalently, we can apply Lemma \ref{pr:hypo2} with $c' (\gamma_c^{\rm I,II})^2$ instead of $c$.
Note that inequality \eqref{rsto} is a direct consequence of \eqref{bb}.

\medskip

For future use in the construction of the approximate solution, we now define rescaled versions of $A_c$ and $A_{c,c'}$ and
we gather   useful information about these functions in the next lemma.
Let 
\begin{align*}
& A_{j,k}(x) = c_j^{\frac 13} A_{  {c_k}/{c_j}}\Big( c_j^{\frac 12} x\Big), \quad
\gamma_{j,k}^{0} = c_j^{\frac 12} \gamma_{c_k/c_j}^{0},\quad 
\gamma_{j,k}^{\rm I,II} = c_j^{\frac 12} \gamma_{c_k/c_j}^{\rm I,II},\quad 
a_{j,k}^{\rm I,II} = c_j^{\frac 13} a_{c_k/c_j}^{\rm I,II},\\
& A_{j,k,l}^{\rm I,II} (x)= c_l^{\frac 13} A_{c_k/c_j,c_j/c_l}^{\rm I,II}\Big(c_l^{\frac 12} x\Big).
\end{align*}

\begin{lemma}\label{le:cns}
Assume 
$$
\frac 34 < \frac {c_k}{c_j} < \frac 43.
$$
\begin{itemize}
\item[{\rm (i)}] For $j<k$, $A_{j,k}$ satisfies 
\be\label{eAjk}
(L_{c_j} A_{j,k})' + \sqrt{c_k} (c_j-c_k)A_{j,k}  = (e^{-\sqrt{c_k} x} Q_{c_j}^3) '.
\ee
\be\label{bjk}
|A_{j,k}(x)|\lesssim e^{\gamma_{j,k}^0} \quad \hbox{for $x<0$},
\ee
\be\label{aAjk}
\left| d^k/dx^k \left (A_{j,k}(x) - a_{j,k}^{\rm I} e^{-\gamma_{j,k}^{\rm I} x} 
-  a_{j,k}^{\rm II} e^{-\gamma_{j,k}^{\rm II} x}\right) \right|\lesssim e^{-\frac 98 |x|}.
\ee
\be\label{gAjk}
\gamma_{j,k}^{0} = \sqrt{c_j - \frac 34 c_k} + \frac 12 \sqrt{c_k} ,\quad 
\gamma_{j,k}^{\rm I} = \sqrt{c_j - \frac 34 c_k} - \frac 12 \sqrt{c_k},\quad 
\gamma_{j,k}^{\rm II} =  \sqrt{c_k}.
\ee
Moreover, 
\be\label{IAjk}
a_{j,k}^{\rm I} \neq 0 \hbox{ and its sign does not depend on $j$ and $k$}.
\ee
\item[{\rm (ii)}]  For $j>k$, $A_{j,k}$ satisfies 
\be\label{eAjkb}
(L_{c_j} A_{j,k})' + \sqrt{c_k} (c_k-c_j)A_{j,k}  = (e^{\sqrt{c_k} x} Q_{c_j}^3) '.
\ee
\be\label{bjkb}
|A_{j,k}(x)|\lesssim e^{\gamma_{j,k}^0} \quad \hbox{for $x<0$},
\ee
\be\label{aAjkb}
\left| d^k/dx^k \left (A_{j,k}(x) - a_{j,k}^{\rm I} e^{-\gamma_{j,k}^{\rm I} x} 
-  a_{j,k}^{\rm II} e^{-\gamma_{j,k}^{\rm II} x}\right) \right| \lesssim e^{-\frac 98 |x|}.
\ee
\be\label{gAjkb}
\gamma_{j,k}^{0} =  \sqrt{c_k}  ,\quad 
\gamma_{j,k}^{\rm I} = \frac 12 \sqrt{c_k}-  \sqrt{c_j - \frac 34 c_k} ,\quad 
\gamma_{j,k}^{\rm II} =  \frac 12 \sqrt{c_k} + \sqrt{c_j - \frac 34 c_k}   .
\ee
Moreover,
\be\label{gkj}
j<k \quad \Rightarrow \quad \gamma_{j,k}^{\rm I} < \gamma_{k,j}^{\rm I}.
\ee
\item[{\rm (iii)}]
For $j\neq k$, $l<j$, $A_{j,k,l}^{\rm I}$ and $A_{j,k,l}^{\rm II}$ satisfy 
\be\label{eAjkl}
(L_{c_l} A_{j,k,l}^{\rm I,II})' + 
\left( \sqrt{c_k} |c_j-c_k|  + \gamma_{j,k}^{\rm I,II} (c_l-c_j) \right)A_{j,k,l}^{\rm I,II}
 = (e^{-{\gamma_{j,k}^{\rm I,II}} x} Q_{c_l}^3) '.
\ee
\be\label{bjkl}
|A_{j,k,l}^{\rm I,II}(x)|\lesssim e^{\sqrt{c_l} x} \quad \hbox{for $x<0$},\quad
|A_{j,k,l}^{\rm I,II}(x)|\lesssim e^{- (\sqrt{c_l} - \gamma_{j,k}^{\rm I,II}) x}, \quad \hbox{for $x>0$}.
\ee
\end{itemize}
\end{lemma}
\begin{proof}[Proof of \eqref{IAjk}] This property follows directly from Lemma \ref{pr:hypo2} (iv).
\end{proof}
\begin{proof}[Proof of \eqref{gkj}]
It is equivalent to prove 
$$
\sqrt{c_j- \frac 34 c_k} - \frac 12 \sqrt{c_j} < \frac 12 \sqrt{c_k} - \sqrt{c_k- \frac 34 c_j},
$$
which is clear since from $c_j>c_k$, we have
$$
\sqrt{c_j- \frac 34 c_k} - \frac 12 \sqrt{c_j}  = \frac {\frac 34 (c_j-c_k)}{\sqrt{c_j- \frac 34 c_k} +\frac 12 \sqrt{c_j}}<\frac {\frac 34 (c_j-c_k)}{\sqrt{c_k- \frac 34 c_j} +\frac 12 \sqrt{c_k}}
= \frac 12 \sqrt{c_k} - \sqrt{c_k- \frac 34 c_j}.
$$
\end{proof}
\begin{proof}[Proof of \eqref{bjkl}]
It is a consequence of \eqref{rsto}. Note   that 
\be\label{gel}
\gamma_{j,k}^{\rm II} <  \sqrt{c_l}.
\ee
Indeed, if $k\leq l$, then $\gamma_{j,k}^{\rm II} < \gamma_{l,k}^{\rm II} \leq \sqrt{c_l}$.
If $j<k<l$, then
$\gamma_{j,k}^{\rm II} < \sqrt{c_k} <  \sqrt{c_l}.$
If $k<j<l$, then  $\gamma_{j,k}^{\rm II}  = \sqrt{c_k}<  \sqrt{c_l}.$
\end{proof}

\subsection{Construction of the approximate solution}\label{sec:V}
In this subsection, we use the functions defined in Lemma \ref{le:cns} to construct an approximate solution.
In this construction, we denote by $E_i$ (for $i=1,\ldots,5$) error terms of size $e^{-2 \sigma_0 t}$.
See Claim~\ref{cE12345}.

\subsubsection{Two soliton interactions}
Inserting
$$R = \sum_j R_j, \quad R_j(t,x) = Q  \left(x-y_j(t)\right), \quad y_j(t)  = c_j t+\Delta_j $$
as a first approximation  into the quartic (gKdV) equation, since $ \partial_t R_j + \partial_x ( \partial_{x}^2 R_j + R_j^4)  
  =   0$, we find
$$
\partial_t R + \partial_x(\partial_x^2 R + R^4)
= \partial_x \bigg( R^4 - \sum_j R_j^4 \bigg)
= \partial_x\bigg( 4 \sum_{j\neq k}  R_k R_j^3+ E_1\bigg) 
$$
where, for some $n_1,n_2,n_3$,
$$
E_1 = n_1 \sum_{ j_1\neq j_2} R_{j_1}^2 R_{j_2}^2 + n_2 \sum_{ j_k\neq j_l} R_{j_1}^2 R_{j_2} R_{j_3}
+ n_3 \sum_{ j_k\neq j_l} R_{j_1} R_{j_2} R_{j_3} R_{j_4}.
$$
The error term $E_1$ is controlled in Claim \ref{cE12345}.

The term $R_k R_j^3$ cannot be considered as an error term. In fact, such term will contribute to the lower bound on the approximate solution which is the key point of the proof of Theorem \ref{th:1}. 
For this term, it is convenient to decouple the variables of $R_k$ and $R_j$ by approximating $R_k$   by its asymptotic expansion around $R_j$(see proof of Claim \ref{cE12345} for more details). Since
$$
4 R_k(t,x)   \mathop{\approx}_{x \sim y_j} 4(10)^{\frac 13} c_k^{\frac 13} e^{-\sqrt{c_k} |x-y_k|} 
= z_{j,k} e^{-\iota_{j,k} \sqrt{c_k} (x-y_j)},
$$
where $\iota_{j,k}=\mathop{\rm sgn}(k-j)$ and
$$
z_{j,k}(t) = 4 (10)^{\frac 13} c_k^{\frac 13} e^{-\iota_{j,k}\sqrt{c_k}(\Delta_j-\Delta_k)} e^{-\sqrt{c_k} |c_j-c_k| t },
$$
 we rewrite the second member of the equation of $R$ as follows
\be\label{fap}
\partial_t R + \partial_x(\partial_x^2 R + R^4)= \partial_x\left(   \sum_{ j\neq k}  z_{j,k} e^{-\iota_{j,k}\sqrt{c_k}(x-y_j)} R_j^3 \right) +\partial_x(
  E_1 + E_2),
\ee
where  $E_2$,  to be controlled in Claim \ref{cE12345},  is the error term generated by this approximation
$$
E_2 = \sum_{  j\neq k}    R_j^3 \left(4 R_k - z_{j,k} e^{-\iota_{j,k}\sqrt{c_k}(x-y_j)}   \right).
$$

\subsubsection{First correction and three soliton interactions}
We define an improved version of the  approximate solution to cancel the  main terms in the right-hand side of \eqref{fap}.
Let 
\be\label{bV}
\bar V := R+Z,
\quad
 Z=\sum_{j,k=1,\ldots,N \atop j\neq k}  Z_{j,k}.
\ee
where
$$Z_{j,k}(t,x)  = z_{j,k}(t)    A_{j,k}(x-y_j(t)),\quad j\neq k.$$
By the equation of $A_{j,k}$ in \eqref{eAjk}, \eqref{eAjkb} we get:
\begin{align*}
 \partial_t Z_{j,k} + \partial_x (\partial_x^2 Z_{j,k} + 4 R_j^3 Z_{j,k})
& =   z_{j,k} \left(- \sqrt{c_k}   |c_j-c_k|  A_{j,k}  - (L_{c_j} A_{j,k})' \right)(x-y_j)\\
& =  - \partial_x\left( z_{j,k} e^{-\iota_{j,k}\sqrt{c_k}(x-y_j)} R_j^3 \right).
\end{align*}
Therefore, using \eqref{fap},
\begin{align}
\partial_t \bar V + \partial_x \left( \partial_x^2 \bar V + \bar V^4\right) & =
\partial_t R + \partial_x \left( \partial_x^2 R + R^4\right) - \partial_x \left(   \sum_{ j\neq k}  z_{j,k} e^{-\iota_{j,k}\sqrt{c_k}(x-y_j)} R_j^3 \right)\nonumber\\Ê& +  \partial_x \bigg( (R+Z)^4 - R^4 - 4 \sum_{j\neq k}R_j^3 Z_{j,k}\bigg)\nonumber
\\ & = \partial_x \bigg( (R+Z)^4 - R^4 - 4 \sum_{j\neq k}R_j^3 Z_{j,k}\bigg) + \partial_x (E_1+E_2)  \nonumber \\
&= \partial_x \bigg( 4 \sum_{j\neq k, j>l} R_l^3 Z_{j,k} \bigg) +\partial_x( E_1+E_2+E_3) .
\label{art}\end{align}
where
$$
E_3 = 4 R^3 Z - 4 \left(\sum_{j\neq k}R_j^3 Z_{j,k} + \sum_{j\neq k, j>l} R_l^3 Z_{j,k}\right)  + 6 R^2 Z^2 + 4 R Z^3 + Z^4.
$$
In the right-hand side of \eqref{art}, the term  $\partial_x\left(\sum_{j\neq k, j>l} R_l^3 Z_{j,k}\right)$ is not an  error term  in the sense that it is   not of size $e^{-2 \sigma_0 t}$.
We use the asymptotic expansion of $A_{j,k}$ in \eqref{aAjk} and \eqref{aAjkb} (depending on $j<k$ or $k<j$) to replace $Z_{j,k}$ in this term  by its asymptotic expansion near $y_l$ (since $j>l$, we have $c_l>c_j$ and $y_l\gg y_j$ for $t$ large)
\begin{align*}
4   Z_{j,k}(t,x)  &\mathop{\approx}_{x \sim y_l} 4 z_{j,k}  a_{j,k}^{\rm I} e^{-\gamma_{j,k}^{\rm I}Ê(x-y_j)} 
+ 4 z_{j,k}  a_{j,k}^{\rm II} e^{-\gamma_{j,k}^{\rm II}Ê(x-y_j)} 
\\ &\mathop{\approx}_{x \sim y_l} z_{j,k}^{\rm I}  e^{-  \gamma_{j,k}^{\rm I}(x-y_l)} + z_{j,k}^{\rm II}  e^{- \gamma_{j,k}^{\rm II} (x-y_l)} 
\end{align*}
where
$$ z_{j,k,l}^{\rm I,II}(t)  =4 z_{j,k}(t) a_{j,k}^{\rm I,II} e^{-\gamma_{j,k}^{\rm I,II} (\Delta_{l}-\Delta_{j})} e^{-\gamma_{j,k}^{\rm I,II}(c_l-c_j)t}.$$
We obtain
\begin{align}
\partial_t \bar V + \partial_x \left( \partial_x^2 \bar V + \bar V^4\right) & 
= \partial_x \left(  \sum_{j\neq k, j>l}    z_{j,k}^{\rm I}  e^{-  \gamma_{j,k}^{\rm I}(x-y_l)}R_l^3 + 
   z_{j,k}^{\rm II}  e^{- \gamma_{j,k}^{\rm II} (x-y_l)}    R_l^3 \right) \nonumber \\
   & + \partial_x\left(\sum_{i=1}^4 E_i\right)\label{sap}
\end{align}
where
$$
E_4 =   \sum_{j\neq k, j>l} \left (4   Z_{j,k} - z_{j,k}^{\rm I}  e^{-  \gamma_{j,k}^{\rm I}(x-y_l)} - z_{j,k}^{\rm II}  e^{- \gamma_{j,k}^{\rm II} (x-y_l)}  \right) R_l^3  .
$$

\subsubsection{Second correction and  final approximate solution}
We refine the above approximate solution $\bar V$   to remove 
the main terms in the right-hand side of \eqref{sap}.
We now define $V$ the final version of approximate solution
\be\label{def:V}
 V  := \bar V + W= R + Z + W , \quad 
 \ee
\be
   W= \sum_{j=2}^N \sum_{k=1\atop k\neq j}^{N}\sum_{l =1}^{j-1} \left(Z_{j,k,l}^{\rm I} + Z_{j,k,l}^{\rm II}\right),
\ee
where
$$ Z_{j,k,l}^{\rm I,II}(t,x)  = z_{j,k,l}^{\rm I,II}(t)    A_{j,k,l}^{\rm I,II}(x-y_l(t)),  \quad j\neq k ,  \  1\leq l <j.$$

First, we observe that for such $V$, \eqref{eq:bar} follows directly from the definition of $\sigma_0$ and $z_{j,k}$, $z_{j,k}^{\rm I,II}$,
$$
\|V(t)- R(t)\|_{H^3} \leq \|Z(t)\|_{H^3} + \|W(t)\|_{H^3}\leq C \sum_{j\neq k}  z _{j,k}(t) \leq C e^{-\sigma_0 t}.
$$

Now, we prove that \eqref{eq:eqV} holds for  $V$.
From the equation of $A_{j,k,l}^{\rm I,II}$ in  \eqref{eAjkl}, we have
 \begin{align*}
& \partial_t Z_{j,k,l}^{\rm I,II} + \partial_x (\partial_x^2 Z_{j,k,l}^{\rm I,II}  + 4 R_l^3 Z_{j,k,l}^{\rm I,II} )
  \\ &=   z_{j,k,l}^{\rm I,II} \left(- \left(\sqrt{c_k}   |c_j-c_k|+\gamma_{j,k}^{\rm I,II} (c_l-c_j)\right)   A_{j,k,l}^{\rm I,II}  - (L_{c_l}    A_{j,k,l}^{\rm I,II})'   \right)(x-y_l)\\
  &  = - \partial_x \left(      z_{j,k}^{\rm I,II}  e^{-  \gamma_{j,k}^{\rm I,II}(x-y_l)}R_l^3\right).
\end{align*}
Therefore, using \eqref{sap}, we get
\begin{align*}
\partial_t V + \partial_x(\partial_x^2 V + V^4)
& = \partial_t \bar V + \partial_x(\partial_x^2 \bar  V + \bar V^4)\\ & - \partial_x \left( \sum_{j\neq k, j>l}
z_{j,k}^{\rm I} e^{-\gamma_{j,k}^{\rm I} (x-y_l)} R_l^3 +z_{j,k}^{\rm II} e^{-\gamma_{j,k}^{\rm II} (x-y_l)} R_l^3 \right)\\
& +  \partial_x\left( V^4 - \bar V^4  -4  \sum_{j\neq k, j>l}  R_l^3 \left(Z_{j,k,l}^{\rm I} + Z_{j,k,l}^{\rm II}\right)\right)\\
& =\partial_x\left(\sum_{i=1}^5 E_i\right),
\end{align*}
where
$$E_5 =  (R+Z+W)^4 - (R+Z)^4  -4  \sum_{j\neq k, j>l}  R_l^3 \left(Z_{j,k,l}^{\rm I} + Z_{j,k,l}^{\rm II}\right).
$$
The following claim (see proof in Appendix \ref{app:D}) completes the proof of \eqref{eq:eqV}.
\begin{claim}\label{cE12345}
For $i=1,\ldots,5$,
$$\|\partial_x E_i(t)\|_{H^3} \leq C e^{-2\sigma_0 t} .$$
\end{claim}

\subsection{Asymptotics  of the approximate solution. Proof of \eqref{eq:lowV}}

First, it is clear that
\be\label{dRj}
0<R_j(t,x) < C e^{-\sqrt{c_j}(x-c_jt)},
\ee
so that
$$
R_j(t,x_0(t)) \lesssim e^{-\sqrt{c_j} (x_0(t) -c_j  t)} .
$$
By \eqref{clbm1} and \eqref{clbm}, we have $\frac {\sigma_0}{\gamma_0} > c_{j_0+1}$ and $\frac {\sigma_0}{\gamma_0}+c_{j_0}
> 5 \sigma_0 + \frac {16}{25}$.
Therefore,
\begin{align*}\sqrt{c_j} (x_0(t) -c_j  t) 
& \geq \frac {\sqrt{3}}2 \left( \frac 4{5 \sqrt{3}} \left(5 \sigma_0 + \frac {16}{25}\right)+ \frac 32 \left( 1- \frac 4{5 \sqrt{3}}\right) - 1\right) t -K_0\\
& \geq \left( 2 \sigma_0 + \frac {\sqrt{3}}2 \left( \frac {64}{125 \sqrt{3}} + \frac 12 - \frac 6{5\sqrt{3}}\right) \right) t -K_0\\ &  \geq 2 \sigma_0 t + \frac 1{10} t - K_0.
\end{align*}
It follows that, for $t>0$,
\be\label{pR}
0< R(t,x_0(t)) < C e^{-\frac 1 {10}t} e^{-2 \sigma_0 t} e^{K_0}.
\ee

Second, for $1\leq j<k\leq N$, we have by \eqref{aAjk}-\eqref{IAjk}, for all $x-c_j t\gg 1$,
\be\label{dZjk}
|Z_{j,k}(t,x)| \geq \kappa e^{-\sqrt{c_k} (c_j-c_k) t} e^{-\gamma_{j,k}^{\rm I} (x-c_jt)}.
\ee
Since all $Z_{j,k}$ for $j<k$ have the same sign at $+\infty$ (see \eqref{IAjk}), their contributions are added and we obtain,
for all $x- t\gg 1$,
\be\label{dZp}
\left| \sum_{j<k} Z_{j,k}(t,x)\right| \geq \kappa \sum_{j<k} e^{-\sqrt{c_k} (c_j-c_k) t} e^{-\gamma_{j,k}^{\rm I} (x-c_jt)}.
\ee
Moreover, using \eqref{gkj}, for $x- c_jt\gg 1$,
\begin{align}
|Z_{k,j}(t,x)| & \leq C e^{-\sqrt{c_j}(c_j-c_k)t} e^{-\gamma_{k,j}^{\rm I} (x-c_kt)}\\
& \leq C e^{-(\sqrt{c_j}-\sqrt{c_k})(c_j-c_k)t}e^{-\sqrt{c_k}(c_j-c_k)t} e^{-\gamma_{j,k}^{\rm I} 
(x-c_jt)}
\\
& \leq C e^{-(\sqrt{c_j}-\sqrt{c_k})(c_j-c_k)t} |Z_{j,k}(t,x)|. 
\end{align}
Thus, for $t$ large enough, and $x -t \gg 1$,
\be\label{dZ}
\left| Z(t,x)\right| \geq \frac {\kappa}2 \sum_{j<k} e^{-\sqrt{c_k} (c_j-c_k) t} e^{-\gamma_{j,k}^{\rm I} (x-c_jt)}.
\ee
and
\be\label{dZz}
\left| Z(t,x_0(t))\right| \geq
\frac {\kappa}2 \sum_{j<k} e^{-\sqrt{c_k} (c_j-c_k) t} e^{-\gamma_{j,k}^{\rm I} (x_0(t)-c_jt)}
\geq  \frac {\kappa}2 e^{\gamma_0 K_0} e^{2\sigma_0 t}.
\ee

Third, to control $W_{j,k,l}$, we use \eqref{bjkl}.
For $j\neq k$, $j>l$, $x-t>0$, we have
$$
\left| Z_{j,k,l}^{\rm I}(t,x)\right| \leq C e^{-\sqrt{c_k} |c_j-c_k| t}
e^{-\gamma_{j,k}^{\rm I}(c_l-c_j)t} e^{-(\sqrt{c_l} -\gamma_{j,k}^{\rm I}) (x-c_l t)}.
$$
In the case where $(c_l-c_j)t > \frac 34 (x-c_l t)$, we obtain
$$
\left| Z_{j,k,l}^{\rm I}(t,x)\right| \leq C e^{-\sqrt{c_k} |c_j-c_k| t} R_l^{\frac 34}(t,x)
\leq e^{-4 \sigma_0 t} + R_l(t,x).
$$ 
In the case where $(c_l-c_j)t < \frac 34 (x-c_l t)$, we obtain
$$
\left| Z_{j,k,l}^{\rm I}(t,x)\right| \leq C e^{-\sqrt{c_k} |c_j-c_k| t} e^{-\sqrt{c_l} (c_l-c_j) t} e^{-\frac 13 (\sqrt{c_l} - \gamma_{j,k}^{\rm I}) (c_l-c_j) t}.
$$

In particular, for $t$ large enough (depending on $K_0$),
we obtain
$$
|R(t,x_0(t)| + |W(t,x_0(t))| \leq \frac {1}{10} |Z(t,x_0(t))|$$
and so
$$
\left| (R+ Z+W)(t,x_0(t)) \right| \geq  
\frac  \kappa 3  e^{\gamma_0 K_0} e^{-2 \sigma_0 t}.
$$

\section{Lower bound for outgoing multi-soliton for $t\gg 1$}\label{sec:4}
In this section, we estimate the distance between the solution $u(t)$ of \eqref{kdv} satisfying \eqref{msm} and the approximate solution $V(t)$ constructed in Section \ref{sec:new4}. The strategy of the proof follows closely Proposition 6 in \cite{Ma2} (see also \cite{MMalmost}).
From \eqref{eq:lowV} and this estimate, we deduce a lower bound on $|u(t,x_0(t))|$ for large time. 

\begin{proposition}\label{le:dist} Assume \eqref{multisolqua} and \eqref{speeds}.
\begin{itemize}
\item[{\rm (i)}] {\rm Comparison with the approximate solution.}
There exists $C>0$ such that for all $t>0$,
\be\label{cp}
\|u(t)-V(t)\|_{H^1} \leq C  e^{-2\sigma_0 t }.
\ee
\item[{\rm (ii)}] {\rm  Lower bound.}
There exist $\kappa_1>0$ and $t_1(K_0)>0$ such that  for $K_0>0$ large enough,
for all $t>t_1(K_0)$, 
\be\label{cq}
|u(t,x_0(t)) |\geq  \kappa_1 e^{\gamma_0 K_0} e^{-2 \sigma_0 t},
\ee
where 
$x_0(t)= \left(\frac {\sigma_0}{\gamma_0} + c_{j_0}\right) t - K_0$.
\end{itemize}
\end{proposition}

\begin{proof} 
{Proof of (i).} The proof is similar to the one of Proposition 6 in \cite{Ma2}, so we only sketch it.
  Let  $$w(t)=u(t)-V(t).$$ 
  Let
$$
R(t,x) = \sum_{j=1}^N R_j(t,x),\quad R_j(t,x)= Q(x-c_j t-\Delta_j).
$$
On the one hand, by   Theorem 1 in \cite{Ma2}, $u(t)\in H^3$, and there exists $\bar \sigma>0$ such that
$$
 \|u(t)- R(t)\|_{H^3} \leq C e^{-\bar \sigma t}.
$$
On the other hand, by \eqref{eq:bar}, 
\be\label{eq:VR}
\|V(t)-R(t)\|_{H^3} \leq C e^{-\sigma_0 t}.
\ee
Thus,    for some $\sigma>0$,
\be\label{wpetit}
\|w(t)\|_{H^3} \leq C e^{-\sigma t}.
\ee

Next, note that $w$ satisfies the following equation
\be\label{eq:w}
w_t + (w_{xx} + (V+w)^4 - V^4)_x +E(V) =0,
\ee
where   $E(V)=\partial_t V + \partial_x(\partial_x^2 V + V^4)$. 
Define 
\begin{align*}
& \mathcal{F}(t)  = \int \left\{ \left(  w_x^2 - \frac 25   \left((V+w)^5 - V^5 - 5 w V^4\right)\right)(t) f(t) + w^2(t)\right\} \\&\hbox{where} \quad
f(t,x)   = \frac 1{c_N} - \sum_{j=1}^{N-1}\left(\frac 1 {c_{j+1}} - \frac 1{c_{j}}\right) \phi\left( \sqrt{\sigma} \left(x- \frac {c_{j+1}+c_{j}}2  t - \frac {\Delta_{j+1}+\Delta_{j}}2\right) \right)
\end{align*}
and  
$$ \sigma = \frac 14 \min(c_1-c_2,  \ldots,c_{N-1}-c_N ,c_N).$$
We claim
\begin{claim}[Energy estimate]\label{cl:F22}
There exist $C, \sigma_1>0$ such that, for $t\geq 0$,
\be\label{eq:F}
\mathcal F(t) \leq C e^{-\sigma_1 t} \sup_{t'\geq t} \|w(t')\|_{H^1}^2 + C e^{-2 \sigma_0 t}  \sup_{t'\geq t} \|w(t')\|_{H^1}.
\ee
\end{claim}
\begin{proof}[Sketch of the proof of Claim \ref{cl:F22}] 
The proof is similar to the one of Lemma 4 \cite{Ma2}. The only difference is the presence of the error term $E(V)$ in \eqref{eq:w}, which generates the second term in the right-hand side of \eqref{eq:F}.

The proof relies on the following estimate of the time derivative of $\mathcal F$: for some $\sigma_1>0$,
\be\label{pf}
\frac {d \mathcal F}{dt}(t) \geq - C e^{-\sigma_1 t} \|w\|_{H^1}^2 -
C e^{-2\sigma_0 t} \|w\|_{L^2}.
\ee
Integrating \eqref{pf} on $[t,+\infty)$, since $\lim_{t\to \infty} \mathcal F(t)=0$ (by \eqref{wpetit}), Claim \ref{cl:F22} is proved.

\medskip

The proof of \eqref{pf} is omitted (see \cite{Ma2}).
We only recall that a key step of the proof is the following property of $V$
\be\label{eq:64}
\|V_t f + V_x\|_{L^\infty} \leq C e^{-\sigma_1 t},
\ee
for some $\sigma_1>0$, easily proved using \eqref{eq:VR} and \eqref{eq:eqV}.
\end{proof}

Next, we claim without proof the following direct consequence of the equations of $w$ and   $Q_{c_j}$.
\begin{claim}[Control of the scaling directions]\label{cl:ortho1}
\be \sum_{j=1}^N \left| \int w(t) R_j(t)\right|  
 \leq C e^{-\sigma_1 t} \sup_{t'\geq t} \|w(t')\|_{H^1}+ C e^{-2\sigma_0t}. 
\ee
\end{claim}

 We now control the translation directions and conclude the proof. Let 
 $$
 \tilde w(t) = w(t) + \sum_{j=1}^N a_j(t) (R_j)_x(t), \ \ 
 a_j(t)= - \frac {\int w(t) (R_j)_x(t)}{\int (R_j)_x^2(t)},\ \
 \int \tilde w(t) (R_j)_x(t) =0,
 $$
 \be\label{compw}
 C_1 \|w(t)\|_{H^1}^2 \leq \|\tilde w(t)\|_{H^1}^2 + \sum_{j=1}^N |a_j(t)|^2\leq C_2 \|w(t)\|_{H^1}^2.
 \ee
We claim  the following result, based on the equations of $(R_j)_x$, 
Claims \ref{cl:F22} and  \ref{cl:ortho1}, 
as well as a coercivity property of $\mathcal F$ up to scaling and translation.
\begin{claim}
For $t>0$, for some $\sigma_1>0$,
\be\label{cl:fi}
	\|\tilde w(t)\|_{H^1}^2 + \sum_{j=1}^N |a_j(t)|^2\leq
	C e^{-\sigma_1 t} \sup_{t'\geq t} \|w(t')\|_{H^1}^2 + Ce^{-4\sigma_0 t}.\ee
\end{claim} 
 From \eqref{cl:fi} and  \eqref{compw}, we obtain 
 $$\|w(t)\|_{H^1}^2 \leq C e^{-\sigma_1 t} \sup_{t'\geq t} \|w(t')\|_{H^1}^2 +C e^{-4\sigma_0 t}  ,$$ and thus, for $t$ large enough,
 $$\frac 12 \|w(t)\|_{H^1}  \leq  Ce^{-2\sigma_0 t} ,$$
 which completes the proof of part (i) of Proposition \ref{le:dist}.

\medskip
 
\noindent{Proof of (ii).} Lower bound.
From \eqref{eq:lowV}  and \eqref{cp}, for $t>t_1(K_0)$,
\begin{align*}
|u(t,x_0(t)) |
	&\geq |V(t,x_0(t))| - |u(t,x_0(t))-V(t,x_0(t))|\\
	& \geq |V(t,x_0(t))| - \|u(t)-V(t)\|_{H^1}\\ 	
	& \geq  \kappa e^{ \gamma_0 K_0} e^{-2 \sigma_0 t} - C e^{-2\sigma_0 t} \geq \frac \kappa2 e^{ \gamma_0 K_0} e^{-2 \sigma_0 t},
\end{align*}
for $K_0$ large enough.
\end{proof}

\appendix
\section{Proof of Lemma \ref{le:mono} and Claim \ref{lun}}\label{ap:B}
\subsection{Proof of Lemma \ref{le:mono}}
 
For $x_0>0$, $t\in [t_1,t_2]$, set the following energy and mass   Liapunov functional:
\begin{align*}
& J_{x_0}(t) = \int \Big( u_x^2 + u^2 - \frac 25 u^5 \Big)(t)\, \psi(t) dx\\
& \hbox{where} \quad \psi(t,x)=\phi\big( \sqrt{\sigma} (x{-}c_0 t {-} (c_0{-}\sigma') (t_2{-}t) {-} x_0)\big) 
\end{align*}
Estimate \eqref{eq:mono4} is based on 
the control of the variation of $J_{x_0}$ on $[t_1,t_2]$.
We claim 
\be\label{surJ}
\frac d{dt} J_{x_0}(t)\leq C e^{-\sqrt{\sigma} (c_0-\sigma') (t_2-t)} e^{-\sqrt{\sigma} x_0} .
\ee
Indeed, we have by direct computations (see e.g. Appendix C in \cite{MMnonlinearity}), 
\begin{align*}
  \frac d{dt} J_{x_0}(t)& = \int \left( - (u_{xx}+u^4)^2 - 2 u_{xx}^2 -3 u_x^2 + 8 u_x^2 u^3 
+ \frac 85 u^5\right) \psi_x \\ &-\sigma' \int  \big( u_x^2 + u^2 - \frac 25 u^5 \big) \psi_x + \int (u_x^2+ u^2) \psi_{xxx}.\end{align*}
Thus, using \eqref{phi3},
\begin{align*}
\frac d{dt} J_{x_0}(t) & \leq \int \left( - (u_{xx}+u^4)^2 - 2 u_{xx}^2 -3 u_x^2 + 8 u_x^2 |u|^3 
+ \frac 25 (4+C_0) |u|^5\right) \psi_x\\&
- (\sigma'-\sigma) \int \big( u_x^2 + u^2\big)\psi_x.
\end{align*}
Note first that $ \|u\|_{L^\infty(x>c_0 t+R)}^2\leq \|u_x\|_{L^2(x>c_0 t+R) } \|u\|_{L^2(x>c_0 t+R) }\leq
C_0 \alpha_0$.
Next,   observe that
\begin{align*}
    \int u_x^2 |u|^3 \psi_x   
  &\leq   \int_{x<c_0 t  +R} u_x^2 |u|^3 \psi_x +    \int_{x>c_0 t  +R} u_x^2 |u|^3 \psi_x
\\ & \leq C \phi'\left( \sqrt{\sigma}  (R{-} (c_0{-}\sigma') (t_2{-}t) {-} x_0)\right)
+ \|u\|_{L^\infty(x>c_0 t+R) }^3\int u_x^2\psi_x
\\ &\leq C e^{-\sqrt{\sigma} (c_0-\sigma') (t_2-t)} e^{-\sqrt{\sigma} x_0}
+C \alpha_0^{\frac 32}  \int u_x^2\psi_x,\end{align*}
and similarly,
$$
\int |u|^5 \psi_x \leq C e^{-\sqrt{\sigma} (c_0-\sigma') (t_2-t)} e^{-\sqrt{\sigma} x_0}
+C \alpha_0^{\frac 32} \int u^2\psi_x.
$$
Estimate \eqref{surJ} follows, for $\alpha_0$ small enough (depending on $\sigma$, $\sigma'$, $c_0$, $C_0$).\\

Integrating \eqref{surJ} on $[t_1,t_2]$, we get
$$
J_{x_0}(t_2)-J_{x_0}(t_1) \leq e^{-\sqrt{\sigma} x_0}.
$$
We control the nonlinear term in $J_{x_0}(t)$ as before:
\begin{align*}
& \int |u|^5 \psi   \leq   \int_{x<c_0 t  +R}   |u|^5 \psi  +    \int_{x>c_0 t  +R}  |u|^5 \psi \\ & \leq C \phi\left( \sqrt{\sigma}  (R{-} (c_0{-}\sigma') (t_2{-}t) {-} x_0)\right)
+ \|u\|_{L^\infty(x>c_0 t+R) }^3\int u^2\psi  \\
& \leq C e^{-\sqrt{\sigma} (c_0-\sigma') (t_2-t)} e^{-\sqrt{\sigma} x_0}
+C \alpha_0^{\frac 32}  \int u^2\psi 
  \leq C   e^{-\sqrt{\sigma} x_0}
+C \alpha_0^{\frac 32}  \int u^2\psi .
\end{align*}
Thus, for $\alpha_0$ small enough,
\begin{align*}
& J_{x_0}(t_2) 
  \geq  \int  \left(u_x^2+u^2    -C |u|^5\right) (t_2)  \psi \\
& \geq \int \left(u_x^2+u^2   \right) (t_2)  \psi  - C   e^{-\sqrt{\sigma} x_0}
-C \alpha_0^{\frac 32}  \int u^2(t_2)\psi 
\\ & \geq \frac 34  \int  (u_x^2+u^2  )(t_2,x) \phi\left(\sqrt{\sigma} (x-c_0 t_2 -x_0) \right)dx
 - C   e^{-\sqrt{\sigma} x_0};
\end{align*}
\begin{align*}
& J_{x_0}(t_1) 
  \leq  \int  \left(u_x^2+u^2    +C |u|^5\right) (t_1)  \psi \\
& \leq \int \left(u_x^2+u^2   \right) (t_1)  \psi  + C   e^{-\sqrt{\sigma} x_0}
+C \alpha_0^{\frac 32}  \int u^2(t_1)\psi 
\\ & \leq \frac 32  \int  (u_x^2+u^2  )(t_1,x) \phi\left(\sqrt{\sigma} (x-c_0 t_1{-} (c_0{-}\sigma') (t_2{-}t_1)-x_0) \right)dx
+ C   e^{-\sqrt{\sigma} x_0}.
\end{align*}

Combining these estimates, we get
\begin{align*}
& \int (u_x^2+u^2  )(t_2,x) \phi\left(\sqrt{\sigma} (x-c_0 t-x_0) \right)dx
\\& \leq
2 \int \big( u_x^2 + u^2   \big)(t_1,x)\, \phi\big( \sqrt{\sigma} (x{-}c_0 t_1{-} (c_0{-}\sigma') (t_2{-}t_1) {-} x_0)\big) dx + C   e^{-\sqrt{\sigma} x_0}.
\end{align*}

\subsection{Proof of Lemma \ref{lun}}
The proof of Lemma \ref{lun} is based on the three conservation laws, mass \eqref{mass}, energy \eqref{energy}  and integral \eqref{int}.

Recall that
$$
\int Q_c^2 = c^{\frac 1{6}} \int Q^2, \quad 
E(Q_c) = c^{\frac 7 {6}} E(Q), \quad \int Q_c = \frac 1{c^{\frac 16}} \int Q,
$$

Let $N\geq 2$, $0<c_N<\ldots<c_1$ and $\Delta_1,\Delta_2,\ldots,\Delta_N\in \RR$.
Let $u(t)$ be the solution of \eqref{kdv} satisfying \eqref{pinfdeux}. Let
$$w(t) = u(t)- \sum_{j=1}^N Q_{c_j} (.- c_j t  -\Delta_j).$$
By the uniqueness result in \cite{Ma2}, we have, for some $C, \ \gamma>0$,
\begin{equation}\label{pinftrois}
 \| w(t)\|_{H^1} \leq Ce^{-\gamma t}.
\end{equation}

First, as in Step 1 of the proof of Lemma \ref{le:droiteN}, using Lemma \ref{le:mono} 
and \eqref{pinftrois},  
\be\label{dd}
\forall x_0>0,\ \forall t>0,\quad
\int_{x<- x_0+ c_N t} u^2(t,x) dx \leq C e^{-\sqrt{\frac {c_N} 2} x_0}.
\ee
Second, following the proof of  Lemma 7.4 in \cite{Munoz}, using the exponential decay in time \eqref{pinftrois}, there exists $C>0$ such that 
\be\label{gg}
\forall t>0,\quad
\int_{x> c_1 t} (x - c_1t)^2 u^2(t,x) dx \leq C
\ee
By the exponential decay properties of $Q_{c_j}$, we deduce from \eqref{dd} and \eqref{gg} that
\be\label{dg}
\forall t>0, \quad
\int_{x<c_Nt} (x-c_N t)^2 w^2(t,x) dx + \int_{x> c_1 t} (x - c_1t)^2 w^2(t,x) dx \leq C.
\ee
From this, we deduce easily that $\int |w(t)|\to 0$ as $t\to +\infty$.
Indeed,
\begin{align*}
&  \int |w(t)|  \leq \int_{x<\frac 12 c_Nt} |w(t)| + \int_{\frac 12 c_Nt<x<\frac 32 c_1t} |w(t)| + \int_{x>\frac 32c_1t} |w(t)| \\
& \leq  C \left(\int_{x<\frac 12 c_Nt} (x{-}c_N t)^{\frac 32} w^2(t)  \right)^{\frac 12}+ C t^\frac 12 \left(\int |w(t)|^2 \right)^{\frac 12}+ 
C \left( \int_{x>\frac 32c_1t} (x {-} c_1t)^{\frac 32} w^2(t)\right)^{\frac 12}\\
& \leq C t^{-\frac 14} + C e^{-\frac \gamma 4 t}.
\end{align*}
\section{Proof of Claim \ref{le:B}}\label{ap:A}
 
Set $D = \frac B {Q^2}$, $E=D Q^{\frac 32} = B Q^{-\frac 12}$, $F = E Q^{\frac 32} = D Q^3= BQ$.
In particular,
$$
\int B Q^{\frac 52}  = \int E Q^{3}  = 0.
$$

 First, we claim:
\begin{claim}\label{cl:comp}
\begin{align}
&\int L(B') B \frac {Q'}{Q^2}=\int (D')^2 (\frac 32 Q^3 + \frac 3{10} Q^6) - \int D^2 (3 Q^3 - \frac {21}5Q^6 +\frac 65 Q^9), \label{LBB}\\
& \int (D')^2 Q^3 = \int (E')^2 + \frac 94 \int E^2 - \frac 95 \int E^2 Q^3,\label{DQ3} \\
& \int (D')^2 Q^6 = \int (F')^2 + 9 \int F^2 - \frac {27}5 \int F^2 Q^3.\label{DQ6}
\end{align}
\end{claim}
Proof of \eqref{LBB}.
\begin{align*}
&\int L(B') B \frac {Q'}{Q^2}= \int (D Q^2 )' L(D Q') =
\int (2 D QQ' + D'Q^2) (D LQ' - 2 D'Q'' - D'' Q') \\
&=\int (D')^2 ( -2  Q^2Q'' + \frac 12 (Q^2Q')' + 2Q(Q')^2)
+ \int D^2 (2 (QQ'Q'')' - (Q(Q')^2)'')\\
& = \int (D')^2 (-\frac 32 Q^2 Q'' + 3 Q (Q')^2) - \int D^2 ((Q')^3)'\\
& =\int (D')^2 (\frac 32 Q^3 + \frac 3{10} Q^6) - \int D^2 (3 Q^3 - \frac {21}5Q^6 +\frac 65 Q^9).
\end{align*}

Proof of \eqref{DQ3}--\eqref{DQ6}. Let $\beta>0$.
\begin{align}
 \int \left[(D Q^{1+\beta})'\right]^2 & =  \int (D')^2 Q^{2 (1+\beta)} + (1+\beta)^2 \int D^2 (Q')^2 Q^{2 \beta} 
+ 2(1+\beta) \int D D'  Q' Q^{1+2 \beta} \nonumber\\
& = \int (D')^2 Q^{2 (1+\beta)} + D^2 \left((1+\beta)^2 (Q')^2 Q^{2 \beta} - (1+\beta) (Q'Q^{1+2 \beta})'  \right) \nonumber\\
& = \int (D')^2 Q^{2(1+\beta)} + D^2 \left(- \beta (1+\beta) (Q')^2 Q^{2 \beta} 
- (1+\beta) Q'' Q^{1+2 \beta}\right)\nonumber\\
&= \int (D')^2 Q^{2(1+\beta)} + (1+\beta) D^2 \left(- (1+\beta)    Q^{2+2 \beta} 
+ \frac {2\beta+5}{5}  Q^{5+2 \beta}\right).\label{genebeta}
\end{align}
Thus, applied to $\beta=\frac 12 $ and $\beta = 2$,
$$
\int (D')^2 Q^3 =  \int (E')^2 + \frac 94 \int E^2 - \frac 95 \int E^2 Q^3
$$
$$ \int (D')^2 Q^6 =   \int (F')^2 + 9 \int F^2 - \frac {27}5 \int F^2 Q^3.
$$

Note also that by definition of $E$ and $F$:
\begin{equation}\label{EF}
\int E^2 Q^3 = \int F^2.
\end{equation}

End of the proof of Claim \ref{le:B}.
We combine \eqref{LBB}-\eqref{DQ3}-\eqref{DQ6}-\eqref{EF} to write $\int L(B') B \frac {Q'}{Q^2}$
as a sum of two nonnegative quadratic forms in $E$ and $F$.
\begin{align*}
 \int L(B') B \frac {Q'}{Q^2}& = \frac 32 \left( \int (E')^2 +  \frac 94 \int E^2 - \frac 95 \int E^2 Q^3\right)\\
&+ \frac 3{10} \left( \int (F')^2 + 9 \int F^2 - \frac {27}{5} \int F^2 Q^3 \right)\\
&- 3 \int E^2  + \frac {21}{5} \int F^2 - \frac 65 \int F^2 Q^3 \\
& = \frac 38 \int E^2 + \frac 32 \left( \int (E')^2 - \frac {27} 5  \int E^2Q^3\right)\\
&+ \frac 3{10} \left( \int (F')^2 + 41 \int F^2- \frac {47}5 \int F^2 Q^3 \right).
\end{align*}

We claim

\begin{claim}\label{cl:pos}
\be\label{forF}
\int (F')^2 + \frac {35} 2  \int F^2- \frac {47}5 \int F^2 Q^3 \geq 0
\ee
\be\label{forE}
\int (E')^2 - \frac {27}5 \int E^2Q^3\geq 0,
\ee
\end{claim}
\begin{proof}
We use standard arguments from \cite{Tit}.
For $\beta>0$, the operator
$$
w'' + \frac 15\beta(2 \beta + 3) Q^3 w
$$
has first eigenfunction $Q^\beta$ and first eigenvalue $-\beta^2$.
In particular, for $\beta=4$, it follows that
\be\label{op}
\forall w,\quad \int (w')^2 + 16 \int w^2 - \frac {44} 5 \int Q^3 w^2 \geq 0.
\ee
Note that this can also be deduced from \eqref{genebeta}.
Since $\frac 35 Q^3 \leq \frac 32$ from the expression of $Q$, we have proved \eqref{forF}.

We also know that for $\frac 32 <\beta\leq 3 $, the operator defined in \eqref{op} has exactly one other eigenfunction
$Q' Q^{\beta-\frac 52}$ with eigenvalue $-(\beta-\frac 32)^2$. In particular, with $\beta=3$,
$$
\int w Q^{3}=\int w Q' Q^{\frac 12} = 0 \quad \Rightarrow \quad \int (w')^2 - \frac {27} 5 \int w^2Q^3\geq 0.
$$
\end{proof}

In conclusion,
$$
\int L(B') B \frac {Q'}{Q^2} \geq \frac 38 \int E^2 + \frac {141}{10} \int E^2 Q^3
\geq \frac 38 \int \frac {B^2}{Q} + \frac {141}{20} \int B^2 Q^2.
$$

\section{Proof of Claim  \ref{cE12345}}\label{app:D}

Here is an elementary claim that we will use frequently in this proof. The proof is immediate and we omit it.
\begin{claim}\label{cl:F}
Let $F_1$ and $F_2$ be two $C^\infty$ functions such that for all $0\leq k\leq 5$, $\sigma_1^+, \sigma_1^-, \sigma_2^+, \sigma_2^->0$, $L>0$,
\be\label{hypoF1}
|F_1^{(k)}|\leq C_k e^{-\sigma_1^+ x}, \hbox{ for $x>0$,} \quad |F_1^{(k)}|\leq C_k e^{\sigma_1^- x}, \hbox{ for $x<0$.}\ee
\be\label{hypoF2}
|F_2^{(k)}|\leq C_k e^{-\sigma_2^+ (x+L)}, \hbox{ for $x>-L$,} \quad |F_2^{(k)}|\leq C_k e^{\sigma_2^- (x+L)}, \hbox{ for $x<-L$.}\ee
Then, for all $k\geq 0$,
\be\label{resultF}\begin{array}{ll}
\hbox{if $\sigma_1^-\neq \sigma_2^+$:} &\| F_1 F_2\|_{H^4} \leq C_k' e^{-\min(\sigma_1^-,\sigma_2^+) L},\\
\hbox{if $\sigma_1^-= \sigma_2^+$:} &\| F_1 F_2\|_{H^4} \leq C_k' L e^{- \sigma_1^-  L},\end{array}
\ee
\end{claim}

Since for all $p\geq 0$,  $|R_j^{(p)}|\lesssim e^{-\sqrt{c_j} (x-c_jt)}$ by Claim \ref{cl:F}, we easily check that 
for $j_k\neq j_l$,
$$
\| R_{j_1}^2 R_{j_2}^2\|_{H^4} \lesssim e^{-2 \sigma_0t},\quad
\| R_{j_1}^2 R_{j_2} R_{j_3}\|_{H^4} \lesssim e^{-3 \sigma_0 t},\quad
\| R_{j_1}  R_{j_2} R_{j_3} R_{j_4}\|_{H^4} \lesssim e^{-4 \sigma_0 t},
$$
and thus
$$
\|E_1\|_{H^4}\lesssim e^{-2 \sigma_0 t}.
$$
Note that for $k\geq 0$,
\begin{align*}
\forall x,\quad 
&\left| Q^{(k)}(x) -    (-{\rm sign}\, x )^{k} (10)^{\frac 13} e^{-|x|} \right| \leq C_k e^{-4|x|},  \\
& \left|Q_{c_j}^{(k)}(x) - (-{\rm sign}\, x )^{k}   (10)^{\frac 13}c_j^{\frac k2 + \frac 13} e^{-\sqrt{c_j} |x|}\right| \leq C_k e^{-4\sqrt{c_j} |x|}.\end{align*}

To estimate $E_2$, we consider two different regions in space.
Assume $k>j$. For $x>y_k:=c_k t + \Delta_k$, we have
\begin{align*}
& \left \|   \left(4 R_k(t,x)  - z_{j,k}(t)  e^{-\sqrt{c_k}(x-y_j)}\right) R_{c_j}^3 \right\|_{H^3(x>y_k)}\\
& = 4 \left\|\left ( Q_{c_k}(x-y_k) -(10)^{\frac 13} c_k^{\frac 13} e^{-\sqrt{c_k} (\Delta_j-\Delta_k)}
e^{-\sqrt{c_k}(c_j-c_k)t}e^{-\sqrt{c_k}(x-y_j)}
\right) R_j^3  \right\|_{H^3(x>y_k)}\\
& = 4 \left\|\left ( Q_{c_k}(x-y_k) -(10)^{\frac 13} c_k^{\frac 13} e^{-\sqrt{c_k} (x-y_k)}
\right) R_j^3  \right\|_{H^3(x>y_k)}\\
& \leq C \left\|e^{-4 \sqrt{c_k}(x-y_k)} e^{-3 \sqrt{c_j} |x-y_j|} \right\|_{L^2(x>y_k)} \lesssim e^{-2 \sigma_0 t}.
\end{align*}
For $x<y_k$, the following is straightforward
$$\left\|z_j  e^{-\sqrt{c_k}(x-y_k)} R_j^3 \right\|_{H^3(x<y_k)} + \left\| R_{k}    R_j^3  \right\|_{H^3(x<y_k)} 
 \lesssim e^{-2 \sigma_0 t}. 
$$
The case $k<j$ is similar. We obtain
$$
 \|E_2\|_{H^4}\lesssim e^{-2 \sigma_0 t}.
$$ 

By the definition of $z_{j,k}$, it is clear that quadratic and higher order terms in $Z$ in the expression of $E_3$ are controlled by
$e^{-2 \sigma_0 t}$, i.e.
$$
\|Z^4\|_{H^4}+ \|R Z^3\|_{H^4} + \|R^2 Z^2\|_{H^4} \leq e^{-2 \sigma_0 t}.
$$
We now consider the remaining term in $E_3$, which we can write as follows
$$
\tilde E_3 = 4 \sum_{j\neq k \atop l_1,l_2,l_3 \not \in L_j} R_{l_1}R_{l_2}R_{l_3} Z_{j,k},
$$
where $L_j = \{(l_1,l_2,l_3) \, |\, l_1=l_2=l_3=j \hbox{ or } l_1=l_2=l_3<j\}$. In the sum defining $\tilde E_3$, if $l_1\neq l_2$, then
$$
\|R_{l_1} R_{l_2} \|_{H^4} \lesssim e^{-  \sigma_0 t},
$$
and thus, the decay of $Z_{j,k}$,
$$
\|R_{l_1}R_{l_2}R_{l_3} Z_{j,k}\|_{H^4} \lesssim e^{-2 \sigma_0 t}.
$$
Therefore, we only have to consider terms such that $l_1=l_2=l_3=l> j$.
For such $j,k,l$, we have immediately, by the decay of $R_l$,
$$
\| R_l^3 Z_{j,k}\|_{H^4(x>y_j)} \lesssim e^{-2 \sigma_0 t}.
$$
For $y<y_j$, we use the space decay of $A_{j,k}$ on the left given in \eqref{bjk} and \eqref{bjkb},
$$
\| R_l^3 Z_{j,k}\|_{H^4(x<y_j)} \lesssim e^{-\sigma_0 t}\| e^{-3 \sqrt{c_l} |x|} e^{-\gamma_{j,k}^{0} |x-(c_j-c_l)t|}\|_{L^2(x<y_j-y_l)} \lesssim e^{-2 \sigma_0 t}.
$$
We have just proved
$$
\|E_3\|_{H^4} \lesssim e^{-2 \sigma_0 t}.
$$

For $E_4$, we use \eqref{bjk}, \eqref{aAjkb} and we argue as before for $E_2$.
$$
E_4 =    \sum_{j\neq k, j>l} \Big (4   Z_{j,k} - z_{j,k}^{\rm I}  e^{-  \gamma_{j,k}^{\rm I}(x-y_l)} - z_{j,k}^{\rm II}  e^{- \gamma_{j,k}^{\rm II} (x-y_l)}  \Big) R_l^3  .
$$
For $j\neq k$, $j>l$,
\begin{align*}
\left\| \Big(Z_{j,k} - z_{j,k}^{\rm I}  e^{-  \gamma_{j,k}^{\rm I}(x-y_l)} - z_{j,k}^{\rm II}  e^{- \gamma_{j,k}^{\rm II} (x-y_l)}  \Big) R_l^3\right\|_{H^4(x>y_j)}
 & \lesssim | z_{j,k}| \left\| e^{-\frac 98 (x-y_j)} e^{-3 \sqrt{c_l} |x-y_l|}\right\|_{L^2(x>y_j)}\\
 & \lesssim e^{-2 \sigma_0 t}.
\end{align*}
The estimate for $y<y_j$ is immediate and we obtain
$$
\|E_4\|_{H^4} \lesssim e^{-2 \sigma_0 t}.
$$

Finally, we consider $E_5$. As for $E_3$, it is clear that quadratic and higher order terms in $W$ in the expression of $E_5$ are controlled by $e^{-2 \sigma_0 t}$, i.e.
$$
\|(R+Z)^2 W^2\|_{H^4} + \|(R+Z) W^3 \|_{H^4} + \|W^4\|_{H^4} \lesssim e^{-2 \sigma_0 t}.
$$
Similary, terms containing products of $Z_{j,k}$ and $Z_{j,k,l}^{\rm I,II}$ are also controlled directly by the expression of $Z_{j,k}$:
$$
\|((R+Z)^3 - R^3 ) W\|_{H^4}\leq e^{-2 \sigma_0 t}.
$$
Therefore, it only remains to estimate the following term
\begin{align*}
& R^3 W - \sum_{j=2}^N \sum_{k=1\atop k\neq j}^{N}\sum_{l =1}^{j-1}  R_l^3 \left(Z_{j,k,l}^{\rm I} + Z_{j,k,l}^{\rm II}\right) \\
& = \sum_{l_1,l_2,l_3} R_{l_1}R_{l_2}R_{l_3}W - \sum_{j=2}^N \sum_{k=1\atop k\neq j}^{N}\sum_{l =1}^{j-1}  R_l^3 \left(Z_{j,k,l}^{\rm I} + Z_{j,k,l}^{\rm II}\right).
\end{align*}
In the first sum on the right-hand side term, when $l_1\neq l_2$ or $l_1\neq l_3$ or $l_2\neq l_3$, 
the corresponding term is immediately controlled:
$$
\left\|\sum_{l_1,l_2,l_3\atop l_n \neq l_{m}} R_{l_1}R_{l_2}R_{l_3}W \right\|_{H^4} \lesssim e^{-2 \sigma_0 t}.
$$
Thus, it only remains to consider terms:
$$
\sum_{l_1=1}^N  \sum_{j=2}^N \sum_{k=1\atop k\neq j}^{N} \sum_{l_2=1\atop l_2\neq l_1}^{j-1} R_{l_1}^3 \left(Z_{j,k,l_2}^{\rm I} + Z_{j,k,l_2}^{\rm II}\right).
$$
To estimate each term of this sum, we distinguish the cases $l_1>l_2$ and $l_2>l_1$.
For $l_1>l_2$, we use the estimate of $A_{j,k,l_2}^{\rm I,II}$ for $x<0$ in \eqref{bjkl}.
Indeed, 
\begin{align*}
& \left\| R_{l_1}^3 Z_{j,k,l_2}^{\rm I,II}\right\|_{H^4(x<y_{l_2})} \lesssim
|z_{j,k,l_2}^{\rm I,II}|\left\| e^{-3 \sqrt{c_{l_1}}|x|} e^{-\sqrt{c_{l_2}} |x-(c_{l_2}-c_{l_1})t|}\right\|_{L^2(x<y_{l_2}-y_{l_1})}
\lesssim e^{-2 \sigma_0 t}.
\end{align*}
and a similar estimate for $x>y_{l_2}$ is clear.

For the case $l_1<l_2$, we argue similarly, but we use the estimate in \eqref{bjkl} for $x>0$ and  
the exact expression of $z_{j,k,l_2}^{\rm I,II}$,
\begin{align*}
  \left\| R_{l_1}^3 Z_{j,k,l_2}^{\rm I,II}\right\|_{H^4(x>y_{l_2})} &\lesssim
|z_{j,k,l_2}^{\rm I,II}|\left\| e^{-3 \sqrt{c_{l_1}}|x|} e^{-(\sqrt{c_{l_2}}-\gamma_{j,k}^{\rm I,II}) (x-(c_{l_1}-c_{l_2})t)}\right\|_{L^2(x>y_{l_1}-y_{l_2})}\\
& \lesssim
e^{-\sqrt{c_k}|c_j-c_k| t} e^{-\gamma_{j,k}^{\rm I,II} (c_{l_2}-c_j) t} e^{-(\sqrt{c_{l_2}}-\gamma_{j,k}^{\rm I,II}) (c_{l_1}-c_{l_2})t}.
\end{align*}
Let $j_1$ be such that $c_{j_1}-c_{j_1+1} = \min_{j} (c_j - c_{j+1}) $.
We have 
$$(\sqrt{c_{l_2}}-\gamma_{j,k}^{\rm I,II}) (c_{l_1}-c_{l_2})\geq  (\sqrt{c_{l_2}}-\gamma_{j,k}^{\rm I,II})
(c_{j_1}-c_{j_1+1}),$$
$$
\gamma_{j,k}^{\rm I,II} (c_{l_2}-c_j) \geq \gamma_{j,k}^{\rm I,II}(c_{j_1}-c_{j_1+1}).
$$
Thus
\begin{align*}
  \left\| R_{l_1}^3 Z_{j,k,l_2}^{\rm I,II}\right\|_{H^4(x>y_{l_2})}
& \lesssim  e^{-\sqrt{c_k}|c_j-c_k| t} e^{-\sqrt{c_{l_2}}  (c_{j_1}-c_{j_1+1})t} 
\lesssim e^{-2 \sigma_0 t}. \end{align*}
The estimate for $x<y_{l_2}$ is clear for this term.

Thus, we have  proved
$$
\|E_5\|_{H^4} \lesssim e^{-2 \sigma_0 t}.
$$
 
\end{document}